\documentclass[leqno,11pt]{amsart}
\newtheorem{thm}{Theorem}[section]
\newtheorem{cor}[thm]{Corollary}
\newtheorem{lemma}[thm]{Lemma}

\theoremstyle{definition}

\theoremstyle{remark}

\newtheorem{remark}[thm]{Remark}
\newtheorem{notation}[thm]{Notation}

\newcommand{\abs}[1]{\lvert#1\rvert}

\newcommand{\R}{\mathbb R}

\newcommand{\C}{\mathbb C}

\renewcommand{\Im}{\operatorname{Im}}

\usepackage{mathrsfs}

\newlength{\intwidth}
\DeclareRobustCommand{\fpint}[2]
   {\mathop{%
      \text{%
              \settowidth{\intwidth}{$\int$}%
              \makebox[0pt][l]{\makebox[\intwidth]{$-$}}%
              $\int_{#1}^{#2}$
           }
           }
   }

\title[Asymptotic representation of minimal polynomials]
{Asymptotic representation of minimal polynomials on several intervals}
\author{F. Peherstorfer}

\begin{document}

\begin{abstract}
Asymptotic representation of minimal polynomials on several intervals is given. The 
last modifications and corrections of this manuscript were done by the author in the 
two months preceding his passing away in November 2009. The manuscript remained unsubmitted 
and is not published elsewhere$^1$. \thanks{$^1$Submitted by P. Yuditskii and I. Moale}
\end{abstract}

\maketitle

\section{Introduction}

Let $E = \bigcup_{k=1}^{l} E_k,$ where $E_k = [a_{2k-1},a_{2k}]$, $a_1<a_2<
\ldots < a_{2l}$, be a system of intervals and let $W \in C(E)$ be a positive weight
function on $E.$ It is a classical problem to find the minimal polynomial on $E$
with respect to a given norm and weight function $W$, that is in case of the
maximum norm, to find the unique monic polynomial $\hat{M}_n(x;W):=\hat{M}_n(x)
= x^n + \ldots $ such that

\[ \begin{split}|| \frac{\hat{M}_n(.;W)}{W}||_\infty:= &\max _{x\in
E}\abs{\frac{\hat{M}_n(x;W)}{W(x)}}\\
= &\min _{a_i} \max _{x\in E} \left|\frac{x^n + a_{n-1}x^{n-1}+ \ldots +
a_0}{W(x)}\right|, \end{split}\] and respectively, to find the normalized minimal
polynomial

\begin{equation}
	M_n(x;W) := \hat{M}_n(x;W)/ E_{n,\infty}(W),
	\label{eq-e1}
\end{equation}
where \[ E_{n,\infty}(w):= || \hat{M}_n(.;W)/W||_\infty \] denotes the minimum
deviation. In the case of a single interval, say $[-1,1]$, and weight function
$W(x) \equiv 1$, it is well known that the Chebyshev polynomial

\begin{equation}
	 2T_n(z) = (\phi(z,\infty,[-1,1]))^n + (\phi(z,\infty,[-1,1]))^{-n}
\label{eq-c1}
\end{equation}
is the normalized minimal polynomial, where
\begin{equation}
\phi(z,\infty,[-1,1]) = z+\sqrt{z^2-1}
\label{eq-c2}
\end{equation}
is the complex Green's function of $\bar{\mathbb C} \backslash [-1,1]$
with pole at infinity, while, even for the single interval case, asymptotic representations with respect
to a weight function have been proved only recently \cite{KroPeh}.

For several intervals not much is known with respect to the $L_\infty$-norm (in
contrast to the $L_2$-norm, that is, for orthogonal polynomials, whose asymptotic
behavior is well understood nowadays thanks to the papers \cite{ref6, Apt, ref42,
Pehzero, Tom, Wid}, see also the forthcoming book \cite{Sim}). Indeed in the thirties
of the last century for two intervals Achieser \cite{Achell} derived an asymptotic
representation of the minimum deviation $E_{n,\infty}(W)$ with the help of
elliptic functions and in the late sixties of the last century Widom \cite{Wid}
found an asymptotic representation of the minimum deviation for several
intervals. But (up to the very special case that $E$ is an inverse image of
$[-1,1]$ under a polynomial mapping, see \cite{Pehdef}) the main points of interest
the explicit or asymptotic representations of the minimal polynomials remained
open, though these open problems have been pointed out in \cite[p. 128, 205]{Wid}.

To state our main results we need some notations. By $\phi(z,z_0)$ we
denote a so-called complex Green's function for $\overline{\mathbb C}
\backslash E_l$ uniquely determined up to a multiplication constant of
absolute value one (chosen conveniently below), that is, $\phi(z,z_0)$
is a multiple valued function which is analytic on $\overline{\mathbb C}
\backslash E$ up to a simple pole at $z = z_0,$ has no zeros on
$\overline{\mathbb C} \backslash E$ and satisfies $|\phi(z,z_0)| \to 1$
for $z \to x \in E$ quasi-everywhere; or in other words $\log |\phi(z,z_0)|$
is the Green's function with pole at $z = z_0 \in \overline{\mathbb C}
\backslash E$, as usual denoted by $g(z,z_0).$ In the case under consideration,
as it is known \cite[§14]{Wid}, a complex Green's function may be represented as
\begin{equation}\label{eq-x1}
    \phi(z,\infty) = \exp \left( \int_{a_{2l}}^{z} r_{\infty}(x) \frac{dx}{\sqrt{H(x)}} \right)
\end{equation}
where
\begin{equation}
    H(x) = \prod\limits_{j=1}^{2l} (x - a_j)
\end{equation}
and $r_{\infty}(x) = x^{l-1} + ...$ is the unique
polynomial such that
\begin{equation}\label{eq-x2}
   \int_{a_{2j}}^{a_{2j+1}} r_{\infty}(x) \frac{dx}{\sqrt{H(x)}} = 0 {\rm \ for \ } j = 0, ..., l-1
\end{equation}
and that for $x \in {\mathbb C} \backslash E$
\begin{equation}\label{eq-x3}
  \phi(z,x_0) = \exp \left( \int_{a_{2l}}^{z} \frac{r_{x_0}(x)}{x - x_0} \frac{dx}{\sqrt{H(x)}} \right)
\end{equation}
where $r_{x_0} \in {\mathbb P}_{l-1}$ is such that $r_{x_0}(x_0) = - \sqrt{H(x_0)}$ and
$$
   \fpint{a_{2j}}{a_{2j+1}} \frac{r_{x_0}(x)}{x - x_0}
   \frac{dx}{\sqrt{H(x)}} = 0 {\rm \ for \ } j = 1, ... , l-1.
$$
Recall that the so-called capacity of $E$ is given by
\begin{equation}\label{cap}
    cap (E) = \lim\limits_{z \to \infty} |\frac{z}{\phi(z,\infty,E)}|.
\end{equation}
By $\omega(z,B;\bar{\mathbb C} \backslash E)$ we denote the harmonic measure of
$B \subseteq E$ with respect to $\bar{\mathbb C} \backslash E$ at $z,$ which is that
harmonic and bounded function on $\overline{\mathbb C} \backslash E_l$ which satisfies
for $\xi \in E_l$ that $\lim\limits_{z \to \xi} \omega(z, B, \overline{\mathbb C}
\backslash E_l) = i_B(\xi),$ where $i_B$ denotes the characteristic function of $B.$
For abbreviation we put $$ \omega(z,E_k;\bar{\mathbb C} \backslash E) = \omega_k(z). $$

Furthermore let us recall that for a given weight function $W$ which is Lipschitz
continuous on $E,$ there is a unique multi-valued analytic function ${\mathcal W}(z)$ with
${\mathcal W}(\infty) > 0$ which has no zeros or poles in $\bar{\mathbb C} \backslash E$
and is such that the limiting values ${\mathcal W}^{\pm}(x) = \lim\limits_{\substack{z \to x \\
\pm {\rm Im \ } z > 0}} {\mathcal W}(z)$ from the upper and lower
half-plane are Lipschitz-continuous on $E$ with the property that
\begin{equation}\label{eq-x4}
    \sqrt{{\mathcal W}^{+}(x) {\mathcal W}^{-}(x)} = W(x) {\rm \ for \ } x \in E.
\end{equation}
Note that
$$ {\mathcal W}(\infty) = \exp \left\{ \frac{1}{2 \pi} \int_E \log W(x) \frac{\partial g(\xi;\infty)}
{\partial n_\xi^{+}} |d\xi| \right\}.$$

\begin{notation} For given $E$ and weight function $W$ put for $k = 1,...,l-1$
\begin{equation}\label{insert_eq}
     \gamma_{k,n}(W) := n \omega_k(\infty) +
     \frac{1}{2 \pi} \int_E \log W(x) \frac{\partial \omega_k(\xi)}{\partial n_\xi^{+}} |d\xi| +
     \sigma_{k,n}
\end{equation}
where $\sigma_{k,n} \in \{ 0,1 \}$ is such that $\gamma_{k,n}(W) \in [0,1]$ modulo $2.$ Note
that $\sigma_{k,n}$ is uniquely determined.

For abbreviation set ${\boldsymbol \omega}(\infty) = (\omega_1(\infty),...,\omega_{l - 1}(\infty))$
and ${\boldsymbol \gamma}_n(W) = $ $(\gamma_{1,n}(W),$ $...,\gamma_{l-1,n}(W)),$ i.e.,
${\boldsymbol \gamma}_n(W) \in [0,1]^{l - 1}$ modulo $2.$
\end{notation}

\begin{thm}\label{thm1}
Let $W \in C^{1+\alpha}(E),$ $\alpha > 0,$ be positive on $E.$ Let $c_{j,n} \in [a_{2j},a_{2j+1}],$
$j=1,...,l-1,$ be the unique points such that
\begin{equation}\label{F1}
    \sum\limits_{j=1}^{l-1} \omega_{k}(c_{j,n}) = \gamma_{k,n}(W) {\rm \ modulo \ } 2 {\rm \ for \ } k = 1,...,l-1,
\end{equation}
and suppose that $\lim\limits_{n \in \mathbb M} c_{j,n} = c_j$ with $c_j \in (a_{2j},a_{2j+1})$
for $j=1,...,l-1,$ where $\mathbb M$ is an infinite subset of $\mathbb N.$ Then the normalized minimal
polynomial $M_n(x;W)$ has for $n \in \mathbb M$ the following uniform asymptotic representation on $E:$
\begin{equation}\label{F2}
   2 M_n(x;W) = \psi_n^{+}(x) +  \psi_n^{-}(x) + o(1),
\end{equation}
where
\begin{equation}\label{F3}
    \psi_n(z) = \frac{\phi(z;\infty)^n}{\prod\limits_{j=1}^{l-1} \phi(z;c_j)} {\mathcal W}(z)
\end{equation}
and the constant in the $o()$ term is independent of $n$ and $x,$ $x \in E.$

Furthermore on any compact subset $K$ of $\bar{\mathbb C} \backslash E$
\begin{equation}\label{F4}
    \frac{2 M_n(z;W)}{\phi(z;\infty)^n} = \frac{{\mathcal W}(z)}{\prod\limits_{j=1}^{l-1} \phi(z;c_j)} + o(1)
\end{equation}
uniformly on $K.$
\end{thm}

\begin{remark}
We note that condition \eqref{F1} implies Widom's condition \cite[Theorem 5.4]{Wid}, that is,
that there exists a unique $l^{\prime} \leq l,$ depending on $n,$ and unique points
$c_{1,n},...,c_{l^{\prime}-1,n}$ from the open gaps such that
\begin{equation}\label{10as}
\sum\limits_{j=1}^{l^{\prime}-1} \omega_{k}(c_{j,n}) = n \omega_k(\infty) + \frac{1}{2 \pi} \int_E \log W(x)
\frac{\partial \omega_k(\xi)}{\partial n_\xi^{+}} |d\xi| {\rm \ \ \ modulo \ } 1.
\end{equation}
In fact we will show that \eqref{F1} and \eqref{10as} are equivalent. \eqref{F1} has the big advantage
that it can be written as a Jacobi-inversion problem from which many informations can be extracted.
\end{remark}

The asymptotic representation \eqref{F4} on compact subsets of $\bar{\mathbb C} \backslash E$
was conjectured by Widom in \cite[p. 205]{Wid}. In fact he conjectured that \eqref{F4} holds
for arbitrary arcs in the complex plane, but this does not hold in general \cite{Pehm}.
Letting $z \to \infty$ in relation \eqref{F4} we obtain by \eqref{cap} immediately the
following expression for the minimum deviation due to Widom \cite[Theorem 11.5]{Wid},
derived by him in a completely different way and for a more general class of weight functions.

\begin{cor}\label{cor1.2}
For $n \in \mathbb M,$
$$ \frac{2 ||M_n(\cdot;W)|| }{(cap E)^n} = \frac{\prod\limits_{j=1}^{l-1}\phi(c_j;\infty)}{\mathcal W(\infty)} + o(1)$$
\end{cor}

We mention that in the case of the three intervals a formula for the capacity
in terms of Theta functions has been given by T. Falliero and A. Sebbar \cite{FalSeb}
recently.

Theorem \ref{thm1} enables us to obtain a precise description of the location of the
zeros of the minimal polynomial in the gaps $(a_{2k},a_{2k+1}), k = 1,...,l-1$ and to
give the number of zeros in the intervals $E_k,$ denoted by $\sharp Z(M_n,E_k),$ as in
the gaps. By the way it is known that $M_n$ has at most one
zero in each gap, which can be proved by Kolmogorov's criteria for the best approximation.

\begin{cor}\label{cor1.3}
Under the assumptions of Theorem \ref{thm1} the following statements about the zeros hold:\\
a) $M_n(z;W)$ has exactly one zero in each interval $[c_j - \varepsilon, c_j + \varepsilon],$
$j = 1,...,l-1, \varepsilon > 0$ for $n \in \mathbb M, n \geq n_0.$ \\
b) For every $n \in \mathbb M, n \geq n_0$ and $k=1,...,l-1$
$$
    \sharp Z(M_n(\cdot;E_k)) = - \sum\limits_{j=1}^{l-1} \omega_k(c_{j,n}) + n \omega_k(\infty)
                             + \frac{1}{2 \pi} \int_E \log W(\xi) \frac{\partial \omega_k(\xi)}{\partial n_\xi}|d \xi|.
$$
\end{cor}

For the study of the limit behavior of the solutions ${\boldsymbol c}_n = (c_{1,n},...,c_{l-1,n})$
of \eqref{F1} it is crucial that the map into the torus $[0,1]^{l - 1}$
\begin{equation*}
\begin{split}
    \tilde{\mathcal A} : & {\sf{X}}_{j = 1}^{l-1} (a_{2j},a_{2j +1}) \to \left( 0, \frac{1}{2} \right)^{l - 1} \\
                         & \ \ \ \ \ (x_1,...,x_{l - 1}) \mapsto
                           \left(
                                  \frac{1}{2} \sum_{j = 1}^{l - 1} \omega_1(x_j),...,
                                  \frac{1}{2} \sum_{j = 1}^{l - 1} \omega_{l - 1}(x_j)
                           \right) \\
\end{split}
\end{equation*}
is a homeomorphismus. In fact it turns out that $\tilde{\mathcal A}$ is the (linearly
transformed) Abel map restricted to the positive sheet of the Riemann surface $y^2 = H.$

We say that $\mbox{\boldmath$c$} \in {\sf{X}}_{j = 1}^{l-1} (a_{2j},a_{2j +1})$ is an accumulation
point of zeros of $\left( M_{n_\nu}(x;W) \right)_{\nu \in \mathbb N}$ if there exists a
sequence $\mbox{\boldmath$y$}_{n_\nu} = $ $(y_{1,n_\nu},...,y_{l-1,n_\nu}) \in $ \newline
${\sf{X}}_{j = 1}^{l-1} (a_{2j},a_{2j +1})$
such that $M_{n_\nu}(y_{j,n_\nu}; W) = 0$ for $j = 1,..., l - 1$ and
$\lim_\nu \mbox{\boldmath$y$}_{n_\nu} = \mbox{\boldmath$c$}.$

\begin{thm}\label{thm2}
The following statements are equivalent:

a) The solutions $\mbox{\boldmath$c$}_{n_\nu} = (c_{1,n_\nu},...,c_{l-1,n_\nu})$
of \eqref{F1} satisfy
\begin{equation*}
    \lim_\nu \mbox{\boldmath$c$}_{n_\nu} = \mbox{\boldmath$c$} {\rm \ with \ }
    \mbox{\boldmath$c$} \in {\sf{X}}_{j = 1}^{l-1} (a_{2j},a_{2j +1})
\end{equation*}

b) $\lim_\nu \mbox{\boldmath$\gamma$}_{n_\nu}(W) = \tilde{\mathcal A}(\mbox{\boldmath$c$})$
modulo $2$ with $\tilde{\mathcal A}(\mbox{\boldmath$c$}) \in (0,1)^{l - 1}.$

c) $\mbox{\boldmath$c$} \in {\sf{X}}_{j = 1}^{l-1} (a_{2j},a_{2j +1})$ is an accumulation
point of zeros of $(M_{n_\nu}(\cdot, W))_{\nu \in \mathbb N}.$

Furthermore, the set of accumulation points of the sequence of solutions
$(\mbox{\boldmath$c$}_n)_{n \in \mathbb N}$ of \eqref{F1} as well as the set
of accumulation points of zeros of \newline
$(M_{n}(\cdot,W))_{n \in \mathbb N}$ in the gaps is dense in
${\sf{X}}_{j = 1}^{l-1} [a_{2j},a_{2j +1}],$ if
$1, \omega_1(\infty),..., \omega_{l - 1}(\infty)$
are linearly independent over $\mathbb Q.$
\end{thm}

If the harmonic measures of the intervals are rational, that is,
\begin{equation}\label{eq-xx0}
   \omega_k(\infty) = \frac{m_k}{N}, \ m_k \in \{ 1,...,N-1 \} {\rm \ for \ } k = 1,...l
\end{equation}
then the points $c_{j,n}$ appear periodically with respect to $n.$ More precisely
we have

\begin{cor}
Suppose that \eqref{eq-xx0} holds and let $b \in \{ 0,...,N-1 \}.$ \\
a) The points $c_{j,n}$ solving \eqref{F1} satisfy
\begin{equation}\label{eq-xx2}
    c_{j,b} = c_{j,b + k N} {\rm \ for \ all \ } k \in \mathbb N
\end{equation}
b) Each $c_{j,b}$ located in the open gap $(a_{2j},a_{2j+1})$ is an
accumulation point of zeros of $(M_{k N + b}(\cdot;W))_{k \in \mathbb N}$ and
there are no other accumulation points of zeros of
$(M_{n}(\cdot;W))_{n \in \mathbb N}$ in the open gaps.
\end{cor}

We mention that the case when the intervals have rational harmonic measure is
covered by Theorem \ref{thm1} also taking into consideration the following remark
and relation \eqref{eq-xx2}.

\begin{remark}
Theorem \ref{thm1} holds true, if some of the $c_{j,n}$'s, $j \in \{ 1,...,l-1 \}$
coincide with a boundary point $a_{2j}$ or $a_{2j + 1}$ of $E$ for each $n \in \mathbb M.$
In this case for the corresponding $c_j$ one has to put $\phi(z, c_{j}) \equiv 1$ in \eqref{F4},
respectively, in Corollary \ref{cor1.2}. Furthermore Corollary \ref{cor1.3} a)
about the zeros holds only with respect to such $c_{j}$'s
which do not coincide with a boundary point of $E.$

Finally we conjecture that the asymptotics from Theorem \ref{thm1} (with the above modifications)
hold true when a boundary point of $E$ is a limit point of the $c_{j,n}$'s.
\end{remark}

There is also an interesting connection between the $L_\infty$-minimum deviation
and minimal polynomials (for that see the Remark at the end of the paper) and that
ones with respect to suitable weighted $L_2$-norm. First we need some notation.
For convenience we set
\begin{equation*}
    1/h(x) :=
    \left\{
       \begin{aligned}
           \frac{(-1)^{l - k}}{\sqrt{|H(x)|}}, & {\rm \ for \ } x \in {\rm int }(E_k) \\
           0, \ \ \                                & {\rm \ elsewhere}
       \end{aligned}
    \right.
\end{equation*}

Let ${\mathcal E} = \{ R : R(x) = x^{l - 1} + ..., R {\rm \ vanishes \ either \ at \ }
a_{2 j} {\rm \ or \ } a_{2 j + 1} {\rm \ for \ }  j = 1,..., l - 1 \}$ and let
${\mathcal E}^{(1 - x)} = \{ (1 - x) R : R \in {\mathcal E} \}.$ Note that for
$R \in {\mathcal E}$ we have that
\begin{equation*}
    \frac{R}{h} = \frac{1}{\sqrt{(x - a_1)(a_{2 l}- x)}}
    \prod_{j = 1}^{l - 1} \left( \frac{x - a_{2 j + 1}}{x - a_{2 j}} \right)^{\varepsilon_j/2}
    > 0 {\rm \ on \ int}(E)
\end{equation*}
and
\begin{equation*}
    \frac{(1-x)R}{h} = \sqrt{\frac{1-x}{1+x}}
    \prod_{j = 1}^{l - 1} \left( \frac{x - a_{2 j + 1}}{x - a_{2 j}} \right)^{\varepsilon_j/2}
    > 0 {\rm \ on \ int}(E)
\end{equation*}
where $\varepsilon_j \in \{ \pm 1 \}.$ The minimum deviation of $x^n$ on $E$ with
respect to the $L_2$-norm squared is denoted by
\begin{equation*}
     E_{n -1, 2}(x^n; W) = \min_{q \in {\mathbb P}_{n - 1}} \int_E |x^n - q(x)|^2 W(x) dx
\end{equation*}
and by $\sharp Z(f;A)$ we denote the number of zeros of $f$ on $A.$

\begin{thm}\label{insertion}
Suppose that the assumptions of Theorem \ref{thm1} ?? hold.

a)
\begin{equation*}
\begin{aligned}
     & E_{2 n - 1, \infty}(x^{2 n}, W) \sim  \\
     & \qquad \frac{1}{2} \max_{\varepsilon_j \in \{ \pm 1 \}}
       E_{n - 1, 2}
       \left(
            x^n; \frac{W(x)}{\sqrt{(x - a_1)(x - a_{2 l})}}
            \prod_{j = 1}^{l - 1}
            \left( \frac{x - a_{2 j + 1}}{x - a_{2 j}} \right)^{\varepsilon_j/2}
       \right)
\end{aligned}
\end{equation*}
where ${\boldsymbol \sigma}_n$ is given by \eqref{} and $R({\boldsymbol \sigma}_n) \in \mathcal E$
is uniquely determined by the condition
\begin{equation*}
    \sharp Z(R({\boldsymbol \sigma}_n),[a_{2j - 1},a_{2 j}]) = \sigma_{j,n}
    {\rm \ modulo \ } 2 {\rm \ for \ } j = 1, ..., l-1
\end{equation*}
b)
\begin{equation*}
\begin{aligned}
    E_{2 n, \infty}(x^{2 n}, W)
    & \sim \frac{1}{2} \max_{\varepsilon_j \in \{ \pm 1 \} }
      E_{n - 1,2}\left(x^n; W(x) \sqrt{ \frac{1 - x}{1 + x}}
      \prod_{j = 1}^{l - 1} \left( \frac{x - a_{2 j + 1}}{x - a_{2 j}} \right)^{\varepsilon_j/2} \right) \\
    & = \frac{1}{2} E_{n - 1,2}(x^n; W(x) ((1 - x)R)({\boldsymbol \sigma}_n)/h)
\end{aligned}
\end{equation*}
where ${\boldsymbol \sigma}_n$ is given by \eqref{} and
$((1-x)R)({\boldsymbol \sigma}_n) \in {\mathcal E}^{(1-x)}$
is uniquely determined by the condition
\begin{equation*}
    \sharp Z((1-x)R ({\boldsymbol \sigma}_n), [a_{2 j - 1},a_{2 j}]) =
    \sigma_{j,n} {\rm \ modulo \ } 2 {\rm \ for \ } j = 1,..., l-1
\end{equation*}
\end{thm}

We note that the asymptotic representation \eqref{F4} and \eqref{F2} may be given
in terms of Theta functions also in the following way.

\begin{cor}\label{cor1.7}
The function $\psi_n/{\mathcal W}$ from Theorem \ref{thm1} may also be written in the form
$$
\begin{aligned}
   c \frac{\psi_n(z)}{{\mathcal W}(z)}
   & = \exp{ \{ - \sum\limits_{k=1}^{l-1}
       ( \frac{1}{2 \pi} \int \log W(\xi) \frac{\partial \omega_k(\xi)}{\partial n_\xi^+} | d \xi | )
       \int_{a_1}^z \varphi_k \} } \\
   &  \qquad \quad \left( \frac{ \vartheta (z; \int_{a_1}^{\infty^{-}} \varphi_k) }
                { \vartheta (z; \int_{a_1}^{\infty^{+}} \varphi_k) } \right)^n
         \ \  \frac{ \vartheta (z; \sum\limits_{j=1}^{l-1} \int_{z_0}^{c_{j,n}^{-}} \varphi_k) }
                { \vartheta (z; \sum\limits_{j=1}^{l-1} \int_{z_0}^{c_{j,n}^{+}} \varphi_k) }
\end{aligned}
$$
\end{cor}

$\vartheta$ denotes the Riemann Theta function, i.e., for given constants $b_j,$
$j = 1,...,l-1$ and a given point $z_0$ on the Riemann surface $\frak{R}$ of
$y^2 = H$
$$ \vartheta(z; b_k) := \vartheta(z; b_1,...,b_{l-1}) :=
   \vartheta (\int_{z_0}^z \varphi_1 - b_1 - k_1, ... , \int_{z_0}^z
   \varphi_{l-1} - b_{l-1} - k_{l-1}) $$
where $k_1,...,k_{l-1}$ are the so-called Riemann-constants and $\varphi_j,$
$j = 1, ..., l-1,$ is a basis of normalized differentials of first kind,
see the beginning of the next Section.

Briefly and roughly speaking we derive the above statements in the following way. First we
consider the case when the weight function is a polynomial, which includes the particular
interesting case $W(x) \equiv 1$ and, what is important, can be treated with the
help of rational functions on the Riemann surface $y^2=H$. More precisely, we write
the system of equations \eqref{F1} in terms of Abelian integrals of first kind, see
\eqref{eq-xx6} below. The transformed system of equations guarantees by Abel's Theorem
the existence of a rational function ${\mathcal R}_n$ on the Riemann surface which is the
keystone in obtaining asymptotics for the minimal polynomials. In fact it turns out that
$R_n:= \mathcal{R}_n^+ + \mathcal{R}_n^-$, is a rational function on $\C$, which equioscillates
$n+1$ times on $E$ and thus has zero as best approximation with respect to any linear subspace
of dimension $n$. Showing that the rational function is asymptotically a polynomial of degree $n$
we obtain, using the Lipschitz continuity of the operator of best approximation
and deriving a so-called strong unicity constant, that this polynomial is asymptotically the
minimal polynomial.

To get the asymptotic representation for general weight functions we approximate
the weight function $W(x)$ by polynomials where it is important that the approximating
sequence of polynomials $(\rho_\nu)$ can be chosen such that the boundary value problem
$I^{+}(x) I^{-}(x) = W(x) / \rho_\nu(x)$ can be solved uniquely by a function $I(z)$
which is a nonvanishing, single valued, analytic function on $\overline{\mathbb C}
\backslash E.$

\section{Asymptotics with respect to rational weights}

In this section we derive an asymptotic formula for the minimal polynomial with respect to rational
weight functions of the form $1/\rho_\nu,$ where $\rho_\nu$ is a polynomial of degree $\nu$ which is positive on $E$,
which on the one hand contains the particular interesting case $W(x) \equiv 1$ (with geometric convergence of the error
even) and on the other hand is the basis to obtain asymptotics with respect to general weight functions roughly speaking
by approximating the weight function by a sequence of $1/\rho_\nu$'s. Hence in the following let
\begin{equation}
    \rho_\nu(x) = \pm \prod\limits_{j=1}^{\nu^{*}} (x - w_j)^{\nu_j} {\rm \ with \ } \rho_\nu > 0 {\rm \ on \ } E.
\end{equation}

         Let $\mathfrak R$ denote the hyperelliptic Riemann surface of genus $l-1$ defined
		 by $y^2 = H(z)$ with branch cuts $[a_1,a_2],
		 [a_2,a_3],\ldots,[a_{2l-1},a_{2l}].$ The two sheets of $\mathfrak R$ are
		 denoted by $\mathfrak R^+$ and $\mathfrak R^-$, where on  $\mathfrak R^+$ the
		 branch of $\sqrt{H(z)}$ is chosen for which $\sqrt{H(x)}>0$\, for $x >
		 a_{2l}$.
		 To indicate that $z$ lies on the first resp. second sheet we write $z^+$ and
		 $z^-$.
		 Furthermore let the cycles  $\{\alpha _j,\beta
		 _j\}_{j=1}^{l-1} $be the usual canonical
		 homology basis  on $\mathfrak R$, i.e., the curve $\alpha _j$ lies on the
		 upper sheet $\mathfrak R^+$ of $\mathfrak R$ and encircles there clockwise the interval
		 $E_j$ and the curve $\beta _j$ originates at $a_{2j}$
		 arrives at $a_{2l-1}$ along the upper sheet and turns back to $a_{2j}$ along
		 the lower sheet. $\mathfrak R' $ denotes now the simple connected canonical
		 dissected Riemann surface.  Let $\{ \varphi _1,\ldots,\varphi_{l-1}\}$,
		 where $\varphi _j = \sum^{l-1}_{s=1} d_{j,s} \frac{z^s}{\sqrt{H(z)}}dz,
		 d_{j,s}\in\C$, be a base of the normalized differential of the first kind i.e.
		 \begin{equation}
		 	\int_{\alpha _j}\varphi _k = 2\pi i\delta _{jk} \quad \text{and}\quad
		 	\int_{\beta_j}\varphi _k = B_{jk} \quad \text{for}\; j,k = 1,\ldots,l-1
		 	\label{eq-38}
		 \end{equation}
		 where $\delta_{jk}$ denotes the Kronecker symbol here. Note that the $d_{j,n}$'s
		 are real since $\sqrt{H(z)}$ is purely imaginary on $E_j$ and since
		 $\sqrt{H(z)}$ is real on $\R \setminus E$ the $B_{jk}$'s are also real. In the
         following $\eta  (P,Q) $ denotes the differential of the third kind which
		 has simple poles at P and Q with residues 1 and -1, respectively, and is
		 normalized such that
		 \begin{equation}
		 	\int_{\alpha_j}\eta  (P,Q) dz = 0\,\quad \text{for}\quad j =1,\ldots,l-1
		 	\label{eq-39}
		 \end{equation}

Next we are going to show that the system of equations \eqref{F1} has a unique
solution. This is done by writing \eqref{F1} in terms of Abelian-integrals
and considering the resulting system as a (real) Jacobi inversion problem,
see Lemma \ref{lemma2.3} a).

We need some preliminaries concerning the connection of harmonic measures and Abelian
integrals more or less known.

\begin{lemma}\label{lem1}
a) Let $c \in \bar{\mathbb C} \backslash E$ and $\bar{c}$ its conjugate number. Then the following
relation holds:
\begin{equation}\label{eq-sec2-x0}
   \int_{c^{-}}^{c^{+}} \varphi_j + \int_{\bar{c}^{-}}^{\bar{c}^{+}} \varphi_j
   = \sum\limits_{k=1}^{l-1} (\omega_k(c)+\omega_k(\bar{c})) B_{j k} {\rm \ \ mod \ \ } (2 B)
\end{equation}
b) Let $\log |W|$ be integrable. Then
\begin{equation*}
      \sum\limits_{\kappa = 1}^{l-1} \left( \frac{1}{2 \pi} \int_E \log |W(\xi)|
        \frac{\partial \omega_\kappa(\xi)}{\partial n^{+}_\xi} |d \xi| \right) B_{k \kappa} \\
    = - \frac{2}{\pi i} \int_{E^+} \varphi_k^{+} \log |W|
\end{equation*}
c)
\begin{equation*}
        \frac{2}{\pi i} \int_{E^+} \varphi_k^{+} \log |\rho_\nu|  =
        \sum\limits_{j=1}^{\nu^{*}} \nu_j ( \int_{\infty^{-}}^{\infty^{+}} \varphi_k -
        \int_{w_j^{-}}^{w_j^{+}} \varphi_k )
\end{equation*}
\end{lemma}

\begin{proof}
a) Let us represent the Green's functions $G(z;c) = \ln \phi(z;c),$ extended
analytically to the Riemann surface by $\phi(\tilde{z};c) = 1/ \phi(z;c)$ in
the form
\begin{equation}\label{eq-sec2-x1}
    d G(z;c) = \sum\limits_{k=1}^{l-1} \mu_k \varphi_k + \eta(z;c^{+},c^{-})
\end{equation}
which implies by integrating along the $\alpha_j$-cycles
\begin{equation}\label{eq-sec2-x2}
    - 2 \pi i \omega_j(c) = \oint_{\alpha_j} d G(z;c) = \sum\limits_{k=1}^{l-1} \mu_k \delta_{jk} 2 \pi i
\end{equation}
where the first equality follows by the representation of the harmonic measure in terms of Green's function,
see e.g. \cite{Neh}, and the last equality by the normalizations of the differentials.

Integrating along the $\beta_j$-cycle and using the bilinear relation for abelian differentials
of the first and third kind, i.e.,
\begin{equation}\label{bil}
   \oint_{\beta_j} \eta(X,Y) = \int_X^Y \varphi_j
\end{equation}
we obtain
\begin{equation}\label{eq-t1}
    \oint_{\beta_j} d G(z;c) = - \sum\limits_{k=1}^{l-1} \omega_k(c) B_{jk} + \int_{c^{-}}^{c^{+}} \varphi_j
\end{equation}
Next let us note that the polynomial $r_c(\xi)$ from \eqref{eq-x3} can be represented in the form
$$ r_c(\xi) = (\xi - c) (r_\infty(\xi) - M_c(\xi)) - \sqrt{H(c)} $$
where $M_c(\xi) = \xi^{l-1} + ... \in {\mathbb P}_{l-1}$ is the unique polynomial which satisfies
$$ \fpint{a_{2j}}{a_{2j+1}} {\frac{\sqrt{H(c)}}{c - x} \frac{d x}{\sqrt{H(x)}} }=
   \int_{a_{2j}}^{a_{2j+1}} M_c(x) \frac{d x}{\sqrt{H(x)}} {\rm \ for \ } j= 1,...,l$$
Hence we obtain that
$$ M_{\bar{c}}(\xi) = \overline{M_c(\xi)}$$
and
$$ d( G(z,c) + G(z,\bar{c}) ) = \frac{r_c(\xi)}{\xi-c} +
   \frac{\overline{r_c(\xi)}}{\xi- \bar{c}} \frac{d\xi}{\sqrt{H(\xi)}}$$
which implies that the integral $\int_{\beta_j} d(G(z,c) + G(z,\bar{c}))$ is purely imaginary,
since the integrand becomes purely imaginary when $\xi$ approaches $E$ and the integrals over
the gaps cancel out by the opposite opposite direction of integration. On the other hand it
follows by \eqref{eq-t1} that the integral is real and thus relation \eqref{eq-sec2-x0} is proved.

b)  Denote by
$$ w_k(z) = \omega_k(z) + i \omega^{*}_k(z) $$
the analytic extension of the harmonic measure on $\overline{\mathbb C}$ and, as usual,
let
\begin{equation}\label{extension}
    w_k = - \omega_k + i \omega_k^{*}
\end{equation}
be the extension to the second sheet. Since $w_k, k = 1,...,l-1,$ is just another basis of Abelian
differentials of first kind we may represent $\int \varphi_j$ as linear combinations of the
$w_k$'s. Integrating along the $\beta_k$ cycle, $ k \in \{ 1,...,l-1 \},$ and recalling the fact
that the integral along a $\beta_k$-cycle is the difference of values along the $\alpha_k$ cycle
(which is $E_k$) we obtain with the help of \eqref{extension} that
\begin{equation}\label{pf-t1}
   \int \varphi_j = - \frac{1}{2} \sum\limits_{k=1}^{l-1} w_k B_{j k}
\end{equation}
Using again the fact that $\omega_k(z) = 1$ on $E_k,$ $k \in \{ 1,...,l \},$ we obtain that
$$ \frac{1}{i} w_k^{'} dz = \frac{1}{i} d w_k(z) = \frac{\partial \omega_k}{\partial n} ds $$
when we approach from the upper half plane to $E,$ hence by \eqref{pf-t1}
\begin{equation}\label{pf-t2}
    - \frac{2}{\pi i} \varphi^{+}_j
    =
      \frac{1}{\pi} \sum\limits_{k=1}^{l-1}
      \left( \frac{\partial \omega_k }{ \partial n^{+} } ds \right) B_{k j}
\end{equation}
from which part b) follows.

Concerning part c) recall that, consider the periods of the harmonic
function $g(\xi,z) - g(\xi, \infty) + \log |\xi - z|,$
$$ \omega_k(z) - \omega_k(\infty) = \frac{1}{2 \pi} \int_E
\frac{\partial \omega_k(\xi)}{\partial n^{+}_\xi} \log | \xi - z| |d \xi|$$
Using \eqref{eq-sec2-x0} and \eqref{pf-t2} the equality follows.
\end{proof}

\begin{notation}
By ${\rm Jac \ } \mathcal R$ we denote the Jacobi variety of $\mathcal R,$
that is, the quotient space ${\mathbb C}^{l - 1} / (2 \pi i \vec{n} + \vec{B} \vec{m}),$
$B = (B_{jk})$ the matrix of periods, $\vec{n},\vec{m} \in {\mathbb Z}^{l - 1}.$
${\rm Jac \ } \mathcal R / {\mathbb R} := {\mathbb R}^{l - 1}/ B \vec{m}$ denotes
the Jacobi variety restricted to the reals. Finally, let $[a_{2 j}, a_{2 j + 1}]^{\pm}$
denote the two copies of $[a_{2j},a_{2 j + 1}],$ $j = 1,...,l-1,$ in ${\mathcal R}^{\pm}.$
Note that $[a_{2j},a_{2j + 1}]^{+} \cup [a_{2j},a_{2j + 1}]^{-}$ is a closed loop on
$\mathcal R.$
\end{notation}

The following important statement holds, see e.g., \cite[Theorem 5.12]{Kre-Lev-Nud}
and \cite[Lemma 3.5 (a)]{PehIMNR}:

The restricted Abel map
\begin{equation*}
\begin{aligned}
    \mathcal A : & {\sf{X}}_{j = 1}^{l-1} ([a_{2j},a_{2j +1}]^{+} \cup [a_{2j},a_{2j +1}]^{-})
                  \to {\rm Jac \ } {\mathcal R} / {\mathbb R}  \\
                 & \ \ \ \ \ (\zeta_1,...,\zeta_{l - 1})
                 \mapsto
                             \frac{1}{2}
                             \left(
                                   \sum_{j = 1}^{l - 1} \int_{\zeta_j^{*}}^{\zeta_j} \varphi_1,...,
                                   \sum_{j = 1}^{l - 1} \int_{\zeta_j^{*}}^{\zeta_j} \varphi_{l - 1}
                             \right) \\
\end{aligned}
\end{equation*}
is a holomorphic bijection. Moreover the so called real Jacobi inversion problem has a unique solution,
that is, for given $(\eta_1,...,\eta_{l-1}) \in {\mathbb R}^{l-1}$ there exists a unique
$(\zeta_1,...,\zeta_{l-1}) \in {\sf{X}}_{j = 1}^{l-1} ([a_{2j},a_{2j +1}]^{+} \cup [a_{2j},a_{2j +1}]^{-})$
such that
\begin{equation*}
    \frac{1}{2} \sum_{j = 1}^{l - 1} \int_{\zeta_j^{*}}^{\zeta_j} \varphi_k = \eta_k {\rm \ \ mod \ } (B).
\end{equation*}

\begin{lemma}\label{lemma2.2}
a) Let ${\boldsymbol s} \in {\mathbb R}^{l - 1}$ and ${\boldsymbol \delta} \in ( \{-1,1\} )^{l - 1}$ be
given. Then there exists a unique $(\kappa_1,...,\kappa_{l-1}) \in \sf{X}_{j=1}^{l-1} {\mathcal R}^{\delta_j}$
and a unique $\vec{\boldsymbol \sigma} \in (\{ 0,1 \})^{l - 1}$ such that
\begin{equation*}
    \frac{1}{2} \sum_{j=1}^{l-1} \int_{\kappa_j^{*}}^{\kappa_j} \varphi_k =
    s_k + \sum_{j=1}^{l-1} \frac{\sigma_j}{2} B_{kj} {\rm \ \ modulo \ } (B) {\rm \ for \ } k = 1,..., l-1
\end{equation*}

b) Choosing in a) $\delta_j = 1$ for $j = 1,...,l-1$ the representation
\begin{equation}\label{ttt0}
    \frac{1}{2} \sum_{j=1}^{l-1} \int_{\kappa_j^{*}}^{\kappa_j} \varphi_k =
    s_k + \sum_{j = 1}^{l - 1} \frac{\sigma_j}{2} B_{kj} = \sum_{j=1}^{l-1} t_j B_{kj} {\rm \ \ mod \ } (B)
\end{equation}
with
\begin{equation}\label{ttt}
    t_\kappa = \frac{1}{2} \sum_{j=1}^{l-1} \omega_\kappa ({\rm pr} (\kappa_j)) \in [0,\frac{1}{2}] {\rm \ for \ }
    \kappa = 1,..., l-1
\end{equation}
holds.

Choosing in a) $\delta_j = -1$ for $j = 1,..., l-1$ then $t_j$ has a representation of the form \eqref{ttt}
with $t_\kappa \in [-\frac{1}{2},0]$ for $\kappa = 1,...,l-1.$
\end{lemma}

\begin{proof}
The simplest way to prove part a) is to change the model of the Riemann surface by choosing
now as in \cite[Section 5.5.2]{Kre-Lev-Nud} the canonical homology basis
$\{ \alpha^{'}_j, \beta^{'}_j \}_{j=1}^{l-1},$ where the curve
$\alpha^{'}_j$ originates at $a_{2j},$ arrives at $a_{2 j + 1}$ along the upper
sheet and turns back to $a_{2j}$ along the lower sheet and $\beta^{'}_j$
lies in the upper sheet and encircles the interval $[a_1,a_{2 j}]$ clockwise.
Note that the $\alpha^{'}$ and $\beta^{'}$ periods can be expressed
easily in terms of the $\alpha$ and $\beta$ periods from \eqref{eq-38}. More precisely,
the period matrix of the new model is obtained by a linear transformation with entire
coefficients of the original model. Thus modulo periods it does not matter
which model we take.

We know that there exists an unique solution $\vec{\xi} \in \sf{X}_{j=1}^{l-1} {\mathcal R},$
see \cite[Theorem 5.12]{Kre-Lev-Nud} and \cite[Lemma 3.5 (a)]{PehIMNR}, such that
\begin{equation}\label{eqx}
  \frac{1}{2} \sum_{j=1}^{l-1} \int_{\xi^{*}_j}^{\xi_j} \varphi_k =
  s_k {\rm \ \ for \ \ } k = 1,..., l-1 {\rm \ \ mod \ \ } B^{'}
\end{equation}
If $\vec{\xi} \in \sf{X}_{j=1}^{l-1} {\mathcal R}^{\delta_j}$ we are done.
If $\xi_{p} \in {\mathcal R}^{- \delta_p}$ then we consider the inversion problem
for $(s_1 + B^{'}_{\frac{p_1}{2}},...,s_{l-1} + B^{'}_{\frac{p_{l-1}}{2}})$
that is, we add the half period $(B^{'})_p.$ Then for $j \neq p$ all
points $\xi_j$ from \eqref{eqx} remain in place and $\xi_p$ will be displaced
by $\xi_p^{*}$ if $\xi_p \notin \{ a_j \}_{j=1}^{2 l}.$ If $\xi_p \in \{ a_j \}_{j=1}^{2 l}$
then this point is displaced from one end of $(a_{2 j},a_{2 j + 1})$ to the other.
Thus part a) follows by going back to our model.

b) By \eqref{} we know that \eqref{ttt0} holds with $t_j \in [-1/2,1/2].$ Taking
into consideration the fact that $\kappa_j = x_j^{+}$ and $\kappa_j^{*} = x_j^{-}$
and writing the LHS of \eqref{ttt0} in terms of harmonic measures with the help of
\eqref{eq-sec2-x0} it follows that
\begin{equation*}
   t_k = \frac{1}{2} \sum_{j=1}^{l-1} \omega_k ({\rm pr}(\kappa_j)) \geq 0
\end{equation*}
which proves part b).
\end{proof}

\begin{lemma}\label{lemma2.3}
a) The unique solution $(\zeta_{1,n},...,\zeta_{l-1,n})$ of the Jacobi inversion
problem, $\kappa = 1,...,l-1,$
\begin{equation}\label{eq-lemma2.3-x}
    \sum_{j=1}^{l - 1} \int_{\zeta_{j,n}^{*}}^{\zeta_{j,n}} \varphi_\kappa
  = n \int_{\infty^{-}}^{{\infty}^{+}} \varphi_\kappa +
    \frac{2}{\pi i} \int_{E^{+}} \varphi_\kappa^{+} \log |W| +
    \sum_{j=1}^{l-1} \sigma_{j,n}(W) B_{\kappa j} {\rm \ \ modulo \ } 2 (B)
\end{equation}
where $\sigma_{j, n}(W)$ is given by \eqref{insert_eq}, has the property that
$\zeta_{j,n} \in {\mathcal R}^{+}$ for $j = 1,..., l-1.$ Moreover,
putting $\zeta_{j,n} = c^{+}_{j,n},$ hence $\zeta^{*}_{j,n} = c^{-}_{j,n}$
and using Lemma \ref{lem1} the system of equations \eqref{eq-lemma2.3-x}
becomes \eqref{F1}.

b) The sequence of solutions $({\boldsymbol c}^{+}_{n_\nu})_{\nu \in \mathbb N}$
of \eqref{eq-lemma2.3-x} converges if and only if
$({\boldsymbol \gamma}_{n_\nu}(W))_{\nu \in \mathbb N}$ converges modulo $2,$
where ${\boldsymbol \gamma}_{n_\nu}(W)$ is given in \eqref{insert_eq}.
Furthermore, the transformed Abel map $\tilde{\mathcal A} = B^{-1} \circ {\mathcal A}$
is a real analytic homeomorphism between the sets of accumulation points
of $({\boldsymbol c}^{+}_{n})_{n \in \mathbb N}$ and
$({\boldsymbol \gamma}_{n}(W))_{n \in \mathbb N}$ modulo $2.$
\end{lemma}

\begin{proof}
Assume that not all $\zeta_{j,n}$'s are from ${\mathcal R}^{+}.$
Then by Lemma \ref{lemma2.2} there is a $\tilde{\boldsymbol \sigma} \in \{0,1 \}^{l-1},$
$\tilde{\boldsymbol \sigma} \neq {\boldsymbol \sigma}_n(W),$
and a $\tilde{\boldsymbol \zeta} \in \sf{X}_{j=1}^{l-1} {\mathcal R}^{+}$ which
solves the correspondingly modified Jacobi inversion problem \eqref{eq-lemma2.3-x}.
Applying Lemma \ref{lem1} to the RHS of \eqref{eq-lemma2.3-x} and using representation
\eqref{ttt0} with property \eqref{ttt} it follows that
${\boldsymbol \gamma}_n(W) - {\boldsymbol \sigma}_n(W) + \tilde{\boldsymbol \sigma} \in [0,1]^{l - 1}$
modulo $2.$ But this implies by the definition of ${\boldsymbol \sigma}_n(W)$ and
its uniqueness, recall \eqref{insert_eq}, that
${\boldsymbol \sigma}_n(W) = \tilde{\boldsymbol \sigma},$
which is a contradiction.

With the help of Lemma \ref{lem1} it follows by straightforward
calculation that \eqref{eq-lemma2.3-x} is equivalent to \eqref{F1}.

b) Writing relation \eqref{eq-lemma2.3-x} in the form
\begin{equation*}
    2 \left( {\mathcal A}(({\boldsymbol c}^{+})_{n_\nu}) \right)^{t} =
    B {\boldsymbol \gamma}^{t}_{n_\nu}
\end{equation*}
where $t$ denotes the transpose of the vectors, the assertion follows immediately
by the bijectivity and the other properties of the Abel map $\mathcal A.$
\end{proof}

\begin{thm}\label{thm1.6}
Let $c_{j,n} \in [a_{2j},a_{2j+1}],$ $j=1,...,l-1,$ be such that
\begin{equation}\label{eq-xx3}
    \sum\limits_{j=1}^{l-1} \omega_k(c_{j,n}) = {\boldsymbol \gamma}_n(1/\rho_\nu)
\end{equation}
and suppose that for $n \in \mathbb M \subseteq \mathbb N$ there holds
$c_{j,n} \in [a_{2j}+\delta, a_{2j+1}-\delta],$ $\delta > 0,$ $j = 1,...,l-1.$ Then
for $n \in \mathbb M$ the following asymptotic representation holds uniformly on $E$
\begin{equation}\label{eq-xx4}
   \frac{2 M_n(x;\rho_\nu)}{\rho_\nu(x)} = \left( {\mathcal R}_n^{+}(x; \rho_\nu) +
   \frac{1}{{\mathcal R}_n^{+}(x; \rho_\nu)} \right) + O(q^n),
\end{equation}
where
\begin{equation}\label{eq-xx5}
    {\mathcal R}_n^{+}(x; \rho_\nu) = \frac{(\phi^{+}(x;\infty))^{n - \nu}
    \prod\limits_{j=1}^{\nu^{*}} (\phi^{+}(x;w_j))^{\nu_j}}{\prod\limits_{j=1}^{l-1} \phi^{+}(x;c_{j,n})}
\end{equation}
and where $q \in (-1,1)$ and the constant in the $O$-term does not depend on $n$ and $x,$
$x \in E.$
\end{thm}

\begin{proof}
For simplicity of writing let us assume that $\nu_j = 1$ for $j = 1,...,\nu.$ First let us transform
condition \eqref{eq-xx3} into the equivalent condition on Abelian differentials of first kind.
Multiplying each equation from \eqref{eq-xx3} by $B_{\kappa k}$ and summing up we obtain that
\eqref{eq-xx3} is equivalent to
\begin{equation}\label{eq-xx6}
\begin{aligned}
    \sum\limits_{j=1}^{l-1} \int_{c_{j,n}^{-}}^{c_{j,n}^{+}} \varphi_\kappa =
        & n \int_{\infty^{-}}^{\infty^{+}} \varphi_\kappa -
          \sum\limits_{j=1}^{\nu}
          \left( \int_{\infty^{-}}^{\infty^{+}} \varphi_\kappa - \int_{w_j^{-}}^{w_j^{+}} \varphi_\kappa \right)\\
        & \qquad + \sum\limits_{k=1}^{l-1} \sigma_{k n}(1/\rho_\nu) B_{\kappa k} {\rm \ \ \ for \ \ \ } \kappa = 1,...,l-1.
\end{aligned}
\end{equation}

Thus by Abel's Theorem, see e.g. \cite[Theorem ]{Osg}, there is a rational function
${\mathcal R}_n$ on the Riemann surface $y^2 = H$ such that
\begin{equation}\label{four.two}
\begin{aligned}
   & x = \infty^{{\pm}} {\rm \ is \ a \ pole \ (zero) \ of \ {\mathcal R}_{n} \ with \ multiplicity \ } n-\nu,\\
   & x = w_j^{\pm} {\rm \ is \ a \ simple \ pole \ (zero) \ of \ } {\mathcal R}_n \\
   & x = c_{j,n}^{{\pm}} {\rm \ is \ a \ simple \ zero \ (pole) \ of \ {\mathcal R}_{n} }
\end{aligned}
\end{equation}
Thus ${\mathcal R}_n$ is of the form
\begin{equation}\label{five}
    {\mathcal R}_{n} = \frac{P_{n+l-1} + \sqrt{H} Q_{n-l-1}}{g_{(n)} \rho_\nu},
\end{equation}
where
\begin{equation}\label{seven}
   g_{(n)}(z) = \prod\limits_{j=1}^{l-1} (x-c_{j,n}),
\end{equation}
$P_{n+l-1}$ and $Q_{n+l-1}$ are polynomials of degree $n+l-1$ and $n-l-1$ which are such that the numerator in
\eqref{five} has the properties that
$$ P_{n+l-1} + \sqrt{H} Q_{n-l-1} {\rm \ has \ a \ double \ zero \ at \ } c_{j,n} {\rm \ for \ } j = 1,...,l-1 $$
and
$$ P_{n+l-1} - \sqrt{H} Q_{n-l-1} {\rm \ has \ a \ double \ zero \ at \ } w_{j} {\rm \ for \ } j = 1,...,\nu. $$
By the way, since the points $c_{j,n}$ and $w_j$ are real the rational function $\overline{{\mathcal R}_n(x)}$
has the same properties \eqref{four.two} as ${\mathcal R}_n(\bar{x}),$ i.e.
${\mathcal R}_n(\bar{x}) = \overline{{\mathcal R}_n(x)},$ or in other words the coefficients of $P_{n+l-1}(x)$
and $Q_{n-1}(x)$ are real. Taking involution (denoted by $\tilde{x}$), in \eqref{four.two}, which corresponds
to multiplication of relation \eqref{eq-xx6} by $-1,$ we obtain that
\begin{equation}\label{eight}
    \frac{1}{{\mathcal R}_{n}(z)} = {\mathcal R}_n(\tilde{z})=
    \frac{P_{n+l-1}(z) - \sqrt{H(z)} Q_{n-l-1}(z)}{g_{(n)}(z) \rho_\nu(z)}
\end{equation}
Moreover, putting
\begin{equation}\label{ten}
    R_n := R := \frac{P_{n+l-1}}{\rho_\nu g_{(n)}} \ {\rm and \ } S_n := S := \frac{Q_{n-l-1}}{\rho_\nu g_{(n)}}
\end{equation}
i.e.
\begin{equation}
    {\mathcal R}_n = R + \sqrt{H} S {\rm \ and \ } \frac{1}{{\mathcal R}_n} = R - \sqrt{H} S
\end{equation}
it follows that
\begin{equation}\label{eleven}
    R^2 - H S^2 = 1
\end{equation}
Note that for $x \in E$
\begin{equation}\label{17.2}
  {\mathcal R}^{\pm}_n(x) = R(x) \pm i \sqrt{- H(x)} S(x)
\end{equation}
hence for $x \in E$
\begin{equation}\label{tilde12}
     2 R_n(x) = {\mathcal R}_n^{+}(x) + {\mathcal R}_n^{-}(x)
              = {\mathcal R}_n^{+}(x) + \frac{1}{ {\mathcal R}_n^{+}(x) },
\end{equation}
where the last equality follows by \eqref{eleven}, and moreover
\begin{equation}
      |{\mathcal R}_n^{\pm}(x)|^2 = {\mathcal R}_n^{+}(x) {\mathcal R}_n^{-}(x) = 1
\end{equation}
Now we claim that
\begin{equation}\label{t4}
   {\mathcal R}_n(z) = (\phi(z;\infty))^{n - \nu}
   \frac{\prod\limits_{j=1}^{\nu} \phi(z;w_j)}{\prod\limits_{j=1}^{l-1} \phi(z;c_{j,n})}
\end{equation}
Indeed, since the function
$$ f(z) = {\mathcal R}_n(z) (\phi(z;\infty))^{-(n-\nu)}
   \frac{\prod\limits_{j=1}^{l-1} \phi(z;c_{j,n})}{\prod\limits_{j=1}^{\nu} \phi(z;w_{j})}$$
has by \eqref{four.two} neither zeros nor poles on $\bar{\mathbb C} \backslash E$ and
satisfies by the definition of $\phi$ and by \eqref{eq-xx3} that $|f^{\pm}| = 1$
on $E.$ Thus $\log |f|$ is a harmonic bounded function on ${\mathbb C} \backslash E$
which has a continuous extension to $E$ and thus $f \equiv 1,$ which proves the claim
\eqref{t4}.

Next let us demonstrate that $|R| \leq 1$ on $E$ and that $R$ has $n+1$ alternation points on $E$
i.e. there exist $n+1$ points $y_i$ from $E,$ $y_1 < y_2 < ... < y_{n+1},$ such that
\begin{equation}\label{eq-5-1}
    (-1)^{n+1-i} = R(y_i) = \frac{P_{n+l-1}(y_i)}{\rho_\nu(y_i) g_{(n)}(y_i)}
\end{equation}
which implies by the Alternation Theorem that
\begin{equation}\label{t5}
   0 {\rm \ is \ a \ best \ approximation \ to \ } R
   {\rm \ with \ respect \ to \ } L \{ x^j/\rho_\nu \}_{j=0}^{n-1}
\end{equation}
Since $- H > 0$ on $\text {\r{E}}$ the property that $|R| \leq 1$ on $E$ follows immediately
by \eqref{eleven}. Furthermore, $|R|=1$ on $E$ if and only if $H Q_{n-1-l} = 0.$ Thus if we are able
to show that the zeros of $H Q_{n-1-l}$ and $P_{n+l-1}$ are simple and strictly interlacing
on $E$ and that $Q_{n-l-1}$ has exactly two zeros in each open gap $(a_{2j},a_{2j+1}),$ $j=1,...,l-1,$
the alternation property \eqref{eq-5-1} and thus \eqref{t5} will follow. First let us recall that
\begin{equation}\label{5-2}
     R = \frac{1}{2} \left( {\mathcal R}_n +  \frac{1}{{\mathcal R}_n} \right) =
     \cosh \ln {\mathcal R}_n {\rm \ and \ } \sqrt{H} S = \frac{1}{2}
     \left( {\mathcal R}_n -  \frac{1}{{\mathcal R}_n} \right) = \sinh \ln {\mathcal R}_n
\end{equation}
Hence
\begin{equation}\label{eq6}
      H Q_{n-l-1}(z) = 0 {\rm \ if \ and \ only \ if \ } {\mathcal R}_n(z) = \pm 1
\end{equation}
where for $z \in E$ one has to take the limiting value ${\mathcal R}_n^{+}.$ Note, that by representation
\eqref{t4}, \eqref{eq-x1} and \eqref{eq-x3}
\begin{equation}\label{t7}
   {\mathcal R}_n^{+}(x) = e^{i \chi_n(x)}
\end{equation}
where
\begin{equation}\label{t72}
\begin{split}
    \chi_n(x) = & - (n - \nu) \pi \int_{a_1}^{x} \frac{r_{\infty}(t)}{h(t)} dt \\
                & + \int_{a_1}^x ({\rm \ bounded \ function \ with \ respect \ to \ } x
                  {\rm \ and \ } n) \frac{dt}{h(t)}
\end{split}
\end{equation}
where we have used the fact that
$$ \lim\limits_{\substack{z \to x \\{\rm Im \ } z > 0}} \frac{r_{\infty}(z)}{\sqrt{H(z)}}
= - \frac{i \pi r_\infty(x)}{h(x)},$$
where
\begin{equation}\label{1.2P}
    \frac{1}{h(x)} = \left\{
        \begin{aligned}
            \frac{(-1)^{l-k}}{\pi \sqrt{-H(x)}} & {\rm \ for \ } x \in E_k \\
          \   0   \qquad \ \                               & {\rm \ elsewhere}
        \end{aligned}
                     \right.
\end{equation}
Since, by \eqref{eq-xx2}, the polynomial $r_{\infty}(x)$ has exactly one zero in each open gap
$(a_{2j},a_{2j+1}),$ $j=1,...,l-1,$ it follows that
\begin{equation}\label{t73}
     |r_\infty(x)| \geq const > 0 {\rm \ on \ } E {\rm \ and \ } r_{\infty}(x) / h(x) > 0 {\rm \ on \ } \text {\r{E}}.
\end{equation}
Thus $\chi_n$ is strictly monotone on $E$ for sufficiently large $n.$ Now by \eqref{5-2} and \eqref{t7}
\begin{equation}\label{eq-19}
   R(x) = \cos \chi_n(x) {\rm \ and \ } i \sqrt{H(x)} S(x) = \sin \chi_n(x),
\end{equation}
hence it follows that on $E$ the zeros of $P_{n+l-1}$ and $H Q_{n-l-1}$ are simple and strictly interlace.

Next let us prove that $Q_{n_k-l-1}$ has exactly two zeros in
$(c_{j}-\varepsilon, c_j+\varepsilon) \subset (a_{2j}, a_{2j+1})$ for $j=1,...,l-1,$ if
$\lim\limits_k c_{j,n_k} = c_j.$

By \eqref{t4} ${\mathcal R}_n$ has a simple zero at $c_{j,n}.$ Since ${\mathcal R}_n$ is real on
$\mathbb R \backslash E$ and since $|\phi(z;\infty)| > 1$ on $\mathbb C \backslash E,$ ${\mathcal R}_n$
is unbounded with respect to $n$ on compact subsets of $(a_{2j},a_{2j+1}) \backslash \{c_j\},$
$j = 1,...,l-1,$ it follows that ${\mathcal R}_n$ takes the values $-1$ and $+1$ on
$(c_j - \varepsilon, c_j + \varepsilon) \subset (a_{2j},a_{2j+1})$ for $j = 1,...,l-1.$
Hence $Q_{n_k-l-1}$ has at least two zeros in $(c_j - \varepsilon, c_j + \varepsilon).$

To show that $Q_{n_k-l-1}$ has exactly two zeros in $(c_j - \varepsilon, c_j + \varepsilon),$
$j = 1,...,l-1,$ we derive first an explicite formula for the number of zeros of $P_{n+l-1}$
in $E_j$ which is of interest by itself. Indeed by \eqref{t4}
\begin{equation}\label{t8}
\begin{aligned}
    d \ln {\mathcal R}_n(z) =
    &(n - \nu) \eta(z; \infty^{-},\infty^{+}) +
        \sum\limits_{\kappa=1}^{\nu} \eta(z; w_\kappa^{-},  w_\kappa^{+} )  \\
    & \ \ \ \ - \sum\limits_{\kappa=1}^{l-1} \eta(z; c_{\kappa,n}^{-},  c_{\kappa,n}^{+} ) +
        \sum\limits_{j=1}^{l-1} e_{j,n} \varphi_j
\end{aligned}
\end{equation}
where we claim that
\begin{equation}\label{t9}
    e_{j,n} = \sharp Z(P_{n+l-1};E_j) := {\rm \ number \ of \ zeros \ of \ } P_{n+l-1} {\rm \ in \ } E_j.
\end{equation}
and that
\begin{equation}\label{t10}
   \begin{aligned}
       & \sharp Z(H Q_{n-l-1}; E_j) - 1 = \sharp Z(P_{n+l-1};E_j) \\
     = & n \omega_j(\infty) - \sum\limits_{\kappa=1}^{l-1} \omega_j(c_{\kappa,n})
         + \frac{1}{\pi} \int_E \log |\rho_\nu(\xi)| \frac{\partial \omega_k}{\partial n^{+}_\xi} |d\xi|
   \end{aligned}
\end{equation}
Indeed, by \eqref{t4} it follows that $d \ln {\mathcal R}_n$ has a representation of the form
\eqref{t8} with $e_{j,n} \in \mathbb C.$ Now by (normalization)
\begin{equation}\label{t11}
    2 \pi i e_{j,n} = \int_{\alpha_j} d \ln {\mathcal R}_n(z)
\end{equation}
and on the other hand by shrinking $\alpha_j$ to $E_j$ and \eqref{t7}, \eqref{ten} and \eqref{eq-19}
we obtain
\begin{equation}\label{t12}
\begin{aligned}
    \int_{\alpha_j} d \ln {\mathcal R}_n(z)
    & = - i \Delta_{E_j} arg {\mathcal R}_n = \chi_n(a_{2j}) - \chi_n(a_{2j-1}) \\
    & = - 2 \pi i \sharp Z(P_{n+l-1};E_j)
\end{aligned}
\end{equation}
where the last equality follows by \eqref{eq-19}, recalling the strictly interlacing property of $P_{n+l-1}$
and $H Q_{n-l-1}.$

Considering the $\beta_k$-cycles gives with the help of \eqref{t9} and Riemann's bilinear relation \eqref{bil}
that
\begin{equation}\label{int}
\begin{aligned}
   \int_{\beta_k} d \ln {\mathcal R}_n(z)
   & = (n - \nu) \int_{\infty^{-}}^{\infty^{+}} \varphi_k +
       \sum\limits_{\kappa=1}^{\nu} \int_{w_{\kappa,n}^{-}}^{w_{\kappa,n}^{+}} \varphi_k -
       \sum\limits_{\kappa=1}^{l-1} \int_{c_{\kappa,n}^{-}}^{c_{\kappa,n}^{+}} \varphi_k \\
   & \quad - \sum\limits_{j=1}^{l-1} \sharp Z(P_{n+l-1};E_j)
\end{aligned}
\end{equation}
Let us consider the integral at the LHS. By the direction of integration the integrals
along the gaps cancel out and along the $E_k$'s the real part of the differential becomes
zero because of \eqref{t4}. Hence the real part of the integral at the LHS \eqref{int} is
zero and thus zero, since the RHS is real. Writing the integrals with the help of the formulas
from Lemma \ref{lem1} a) and b) it follows that \eqref{t10} holds.

Since $\sum\limits_{k=1}^l \omega_k(z) = 1$ we obtain by summing up \eqref{t10} that
$$ n-(l-1) = \sharp Z(H Q_{n-l-1};E) - l$$
i.e. $H Q_{n-l-1}$ has $n+1$ zeros in $E$ (which implies that $Q_{n_k-l-1}$ has at most
and thus exactly two zeros in each interval $(c_j - \varepsilon, c_j + \varepsilon),$
$j= 1,...,l-1$) and thus $R$ has $n+1$ alternation points on $E$ which proves statement
\eqref{t5}.

Next let us show that $P_{n+l-1}/g_{(n)}$ is asymptotically equal on
$\Omega \backslash \{ \bigcup\limits_{j=1}^{l-1} U_\varepsilon(c_j) \}$
to a polynomial $\tilde{M}_n$ of degree $n.$ Indeed partial fraction expansion gives
\begin{equation}\label{M1}
\frac{P_{n+l-1}(z)}{g_{(n)}(z)} = \tilde{M}_n(z) + \sum\limits_{j=1}^{l-1} \frac{\lambda_{j,n}}{z - c_{j,n}}
\end{equation}
Thus
\begin{equation}\label{21-2}
\begin{aligned}
  \lambda_{j,n} & = \lim\limits_{z \to c_{j,n}} (z - c_{j,n}) \frac{P_{n+l-1}(z)}{g_{(n)}(z)} \\
                & = \lim\limits_{z \to c_{j,n}} (z - c_{j,n})
                    \frac{ \prod\limits_{j=1}^{l-1} \phi(z; c_{j,n},E) }{\phi(z; \infty,E)^n}
                    \frac{\rho_\nu(z)}{\prod\limits_{j=1}^{\nu} \phi(z;w_j)}\\
  \end{aligned}
\end{equation}
i.e.
\begin{equation}\label{est}
    \lambda_{j,n} = O(q^n), \ q < 1,
\end{equation}
where we used the facts that $|\phi(z,\infty)| \geq 1/q$
on $\Omega$ and that the $c_{j,n}$'s stay away from this set.

Finally let us show that asymptotically $0$ is a best
approximation to $M_n(\cdot;\rho_\nu)/\rho_\nu$ by demonstrating that the best
approximation of $R_n = P_{n+l-1}/\rho_\nu g_{(n)}$ and $M_n(\cdot;\rho_\nu)/\rho_\nu$
differ only slightly from each other for $n$ sufficiently large. To prove this fact
we use the so-called strong unicity constant defined for a function $f \in
C(E)$ with respect to a linear space $G_n$ by
\begin{equation}\label{e1}
 \gamma(f; G_p) := \gamma(f) := \sup\limits_{p \in G_n} \frac{|| p-p^{*}(f) ||}{||f-p||-||f-p^{*}(f)||}
\end{equation}

If $G_n$ is a Haar system on $E$ of dimension $n$ and $f - g^{*}$ has the alternation points
$y_1,...,y_{n+1},$ then it is known that
\begin{equation}\label{t-2}
     \gamma(f) \leq \max\limits_{1 \leq k \leq n+1} ||p_k||
\end{equation}
where the $p_k$'s from $G_n$ are uniquely defined by
\begin{equation}\label{t-1}
    p_k(y_k) = sgn (f - p^{*}(f))(y_k) =: \sigma_k {\rm \ and \ } p_k(y_j) = 0 {\rm \ for \ } j \neq k
\end{equation}
Furthermore, let us recall that the operator of best approximation is Lipschitz continuous that is,
if $h \in C(E)$ is another function, then
$$ ||p^{*}(f) - p^{*}(h)|| \leq 2 \gamma(f) ||f - h||. $$
In the case under consideration we put $f = R_n,$ $G_n = L \{ x^j/\rho_\nu(x) \}_{j=0}^{n-1},$
$h = \tilde{M}_n/\rho_\nu$ which yields by \eqref{t5} and \eqref{21-2} that
$$ ||p^{*}(\tilde{M}_n)|| \leq 2 \gamma(R_n) O(q^n) $$
Since we will prove in the final step that
\begin{equation}\label{tt}
    \gamma(R_n) = O(n)
\end{equation}
it follows by \eqref{21-2}, recall that $||R_n|| = 1,$ that
$$ || \frac{\tilde{M}_n}{\rho_\nu} - p^{*}(\tilde{M}_n)|| = 1 + O(q^n) $$
and thus
\begin{equation}\label{62_2}
\begin{aligned}
    \frac{M_n(x)}{\rho_\nu(x)}
    & = \frac{\tilde{M}_n(x) - p^{*}(x; \tilde{M}_n)} {\rho_\nu(x) || \frac{\tilde{M}_n}{\rho_\nu} - p^{*}(\tilde{M}_n)||}
      = \frac{\tilde{M}_n(x)}{\rho_\nu(x)} + O(q^n) \\
    & = R_n(x) + O(q^n)
\end{aligned}
\end{equation}
which is by \eqref{tilde12} the assertion of the theorem.

Thus let us prove \eqref{tt}. First we note that in the case under consideration
$$ p_k(y) = \sigma_k \rho_\nu(y_k) \frac{l_{k,n}(y)}{\rho_\nu(y)} $$
where $l_{k,n}$ is the fundamental Lagrange polynomial with respect to
the nodes $y_1,...,y_{n+1}$ which, as we have proved above, are by \eqref{eq-5-1},
\eqref{eleven} and \eqref{ten} the zeros of
\begin{equation}\label{E10-t0}
    H \tilde{Q}, {\rm \ where \ } Q_{n-l-1} = \tilde{Q} v_{2l - 2, n}
\end{equation}

Recall that we have shown that
\begin{equation}\label{E10-t1}
   v_{2l-2,n}(x) \underset {n \to \infty} \longrightarrow \prod\limits_{j=1}^{l-1} (x - c_j)^2
   {\rm \ and \ that \ } g_{(n)}(x) \underset {n \to \infty} \longrightarrow \prod\limits_{j=1}^{l-1} (x - c_j).
\end{equation}
Thus by \eqref{t-2}
$$ \gamma(R_n) \leq const \max\limits_{1 \leq k \leq n+1} ||l_{k,n}(y)||$$
and therefore it suffices to show that
\begin{equation}
    ||l_{k,n}(y)|| = || \frac{(H \tilde{Q})(y)}{(y-y_k)(H \tilde{Q})'(y_k)} || = O(n)
\end{equation}
With the help of the mean value and Markov's inequality we get that
$$ ||\frac{H \tilde{Q}(y)}{y-y_k}|| \leq n^2 const ||H \tilde{Q}||
                                    \leq n^2 \widetilde{const} || \sqrt{H} \frac{Q_{n-l-1}}{g_{(n)}} || = O(n^2)$$
where we took into consideration \eqref{E10-t1} and \eqref{eleven}. Finally let us show
that at the zeros $y_k$ of $H \tilde{Q}$
\begin{equation}\label{E10-x0}
    | (H \tilde{Q})'(y_k) | \geq const \ n
\end{equation}
Using \eqref{ten}, \eqref{eq-19} and \eqref{E10-t0} it follows that at the zeros
$y_k$ of $\tilde{Q}$
$$ (H Q_{n-l-1})' = \pm ( \sqrt{- H} \rho_\nu g_{(n)} \chi^{'}_n)(y_k) $$
where we used the fact that $\cos \chi_n(y_k) = \pm 1.$ By \eqref{t72}, \eqref{t73}
and \eqref{E10-t0} inequality \eqref{E10-x0} follows at the zeros of $\tilde{Q}.$

At a boundary point of $E,$ say $a_j,$ we have
\begin{equation}\label{E10-x2}
     (H Q_{n-l-1})' = \rho_\nu g_{(n)} H' Q_{n-l-1}
\end{equation}
Now by \eqref{t72} and \eqref{eq-19}, recall that $\sin \chi_n(a_j) = 0,$
$$ \lim\limits_{\substack{x \to a_j \\ x \in E}} \frac{Q_{n-l-1}}{\rho_\nu g_{(n)}} =
   \lim\limits_{x \to a_j} \frac{\sqrt{-H(x)} \sin \chi_n(x)}{-H(x)} $$
which implies, with the help of $|r_{\infty}(a_j)| > 0,$ that \eqref{E10-x0} holds at
the zeros of $H$ also.
\end{proof}

\begin{cor}
For the minimum deviation the following asymptotics hold for $n \in \mathbb M$
\begin{equation}
   E_{n-1}(x^n;\rho_\nu) = (cap E)^{n-\nu}
   \frac{\prod\limits_{j=1}^{l-1} \phi(c_{j,n};\infty)}{\prod\limits_{j=1}^{\nu^{*}} \phi(w_j;\infty)^{\nu_j}}
   (1 + O(q^n))
\end{equation}
\end{cor}

\begin{proof}
Recalling that
$$ R_n = \frac{P_{n+l-1}}{\rho_\nu g_{(n)}} = \frac{1}{2} \left( {\mathcal R}_n + \frac{1}{{\mathcal R}_n}\right) $$
it follows by \eqref{t4} and $|\phi| > 1$ on $\mathbb C \backslash E$ that
\begin{equation}\label{E11-x1}
\begin{aligned}
  2 lc (P_{n+l-1}) & = 2 \lim\limits_{z \to \infty} \frac{P_{n+l-1}}{z^{n-\nu} \rho_\nu g_{(n)}}
                     = \lim\limits_{z \to \infty} \frac{{\mathcal R}_n}{z^{n-\nu} } + O(q^n) \\
                   & = \lim\limits_{z \to \infty} \left( \frac{\phi(z;\infty)}{z} \right)^{n-\nu}
                       \frac{\prod\limits_{j=1}^{\nu^{*}} \phi(w_j;\infty)^{\nu_j}}
                       {\prod\limits_{j=1}^{l-1} \phi(c_{j,n};\infty)}
\end{aligned}
\end{equation}
Since by \eqref{eq-5-1} at the zeros $y_i$ of $H \tilde{Q},$ $\tilde{Q}$ defined in \eqref{E10-t0},
$$ (-1)^{n+1-i} = R_n(y_i) = \frac{\tilde{M}_n(y_i)}{\rho_\nu(y_i)} + O(q^n) $$
it follows by Vall\'{e}e-Poussin's Theorem, see \cite{},
\begin{equation}\label{poussin}
    lc(\tilde{M}_n) E_{n-1}(x^n; \rho_\nu) = 1 + O(q^n)
\end{equation}
Because of \eqref{M1} we have that
\begin{equation}\label{E11x2}
    lc(P_{n+l-1}) = lc(\tilde{M}_n)
\end{equation}
which gives by \eqref{E11-x1} and \eqref{poussin} the assertion.
\end{proof}

\section{Proof of Theorem \ref{thm1}}

The link with the weights of the form $1/\rho$ is given by the following two Lemmatas.
For the next lemma compare \cite{KroPeh}.

\begin{lemma}\label{lemma-II}
Let $W \in C(E)$ and $\rho_\nu$ be positive on $E.$ Then
\begin{equation}\label{l2.2-t}
    || \frac{M_n(x;W)}{W(x)} - \frac{M_n(x;\rho_\nu)}{\rho_\nu(x)} || =
    O(n) \left( O(|| 1 - \frac{\rho_\nu}{W}||) + O(q^n) \right)
\end{equation}
where $q \in (-1,1).$
\end{lemma}

\begin{proof}
Put
\begin{equation}\label{LO}
    a_n = 1/ E_{n-1}(x^n;W) {\rm \ and \ } b_n = 1/E_{n-1}(x^n;\rho_\nu)
\end{equation}
that is, $M_n(x;W) = a_n x^n + ...$ and $M_n(x;\rho_\nu) = b_n x^n + ... .$
Using the extremal property of $M_n(\cdot;W)$ and $M_n(\cdot;\rho_\nu)$
we obtain that
$$ \frac{1}{a_n} \leq || \frac{M_n(x;\rho_\nu)}{W} || \leq \frac{1}{b_n} ||\frac{\rho_\nu}{W}||$$
and an analogous estimate for $M_n(x;W)/\rho_\nu$ yields
$$ 1/ ||\frac{\rho_\nu}{W}|| \leq \frac{a_n}{b_n} \leq ||\frac{W}{\rho_\nu}||$$
which implies using
$$ || \frac{W}{\rho_\nu} - 1 || \leq ||\frac{W}{\rho_\nu}|| ||\frac{\rho_\nu}{W} - 1||
   {\rm \ and \ } 1 \leq ||\frac{\rho_\nu}{W}|| ||\frac{W}{\rho_\nu}|| $$
that
\begin{equation}\label{L1}
    \left| \frac{a_n}{b_n} - 1 \right| \leq ||\frac{W}{\rho_\nu}|| ||\frac{\rho_\nu}{W} - 1||
\end{equation}
Let
$$ R_n(x;\rho_\nu) := R_n(x) = \frac{P_{n+l-1}}{\rho_\nu g_{(n)}} $$
be given by \eqref{ten}. Since $||M_n(x;W)/W|| = 1$ and by \eqref{62_2}
$$ || \frac{M_n(x;\rho_\nu)}{\rho_\nu} - R_n(x;\rho_\nu)|| = O(q^n) $$
it suffices to show that
$$ || \frac{M_n(x;W)}{\rho_\nu} - R_n(x;\rho_\nu) || = O(n)
   \left( O(|| 1 - \frac{\rho_\nu}{W} ||) + O(q^n) \right) $$
We may write
$$ M_n(x;W) - \frac{P_{n+l-1}(x)}{g_{(n)}(x)} = (1 - \frac{a_n}{b_n}) M_n(x;W) +
M_n(x;\rho_\nu) - \frac{P_{n+l-1}(x)}{g_{(n)}(x)} + h(x) $$
where
$$ h(x) = \frac{a_n}{b_n} M_n(x;W) - M_n(x;\rho_\nu) \in {\mathbb P}_{n-1} $$
Thus by \eqref{ten}
$$
\begin{aligned}
     & || - \frac{h}{\rho_\nu} || \leq \gamma(R_n) \left( ||R_n +
       \frac{h}{\rho_\nu}|| - ||R_n|| \right) \\
   = & \gamma(R_n) \left( || R_n - \frac{M_n(x,\rho_\nu)}{\rho_\nu} -
       (1 - \frac{a_n}{b_n}) \frac{M_n(x;W)}{\rho_\nu} +
       \frac{M_n(x;W)}{W} \frac{W}{\rho_\nu}|| - 1 \right) \\
   = & \gamma(R_n) \left( O(q^n) + O(||\frac{\rho_\nu}{W} - 1||) \right) \\
\end{aligned}
$$
which gives by \eqref{l2.2-t} the assertion.
\end{proof}

The following lemma is due to Achieser and Tomcuk \cite{ref6} and shows that
the RHS in \eqref{l2.2-t} tends to zero, if $W \in C^{1 + \alpha},$ $\alpha > 0.$

\begin{lemma}\label{lemma-I}
Let $W \in C^m(E)$ be positive on $E$ with
$\lim\limits_{n \to \infty} \omega_2(\frac{1}{n}) \log n = 0$ where $\omega_2$ denotes the modulus of
continuity of second order. Then there is a sequence of polynomials $\rho_\nu$ of degree $\nu,$
$\rho_\nu$ positive on $E,$ such that
\begin{equation}\label{II-1-t1}
    |\frac{\rho_\nu(x)}{W(x)} - 1| \leq \frac{const}{\nu^m} \omega_2(\frac{1}{\nu})
\end{equation}
and
\begin{equation}\label{II-1-t2}
    \int_E \log |\rho_\nu(\xi)| \frac{\partial \omega_k(\xi)}{\partial n^{+}_{\xi}} |d \xi| =
    \int_E \log |W(\xi)| \frac{\partial \omega_k(\xi)}{\partial n^{+}_{\xi}} |d\xi|
    {\rm \ for \ } k = 1,...,l-1
\end{equation}
\end{lemma}

\begin{proof}
In \cite{ref6} instead of relation \eqref{II-1-t2} the relation
\begin{equation}\label{II-2-t3}
    \int_E x^j \log \frac{\rho_\nu(x)}{W(x)} \frac{dx}{h(x)} = 0 \ \ \ j = 0,...,l-1
\end{equation}
is given which obviously is equivalent to
\begin{equation}\label{II-2-t4}
    \int_E \varphi_j^+ \log \frac{\rho_\nu(x)}{W(x)} = 0 \ \ \ j = 0,...,l-1
\end{equation}
Now \eqref{II-2-t3} follows by \eqref{pf-t2}.
\end{proof}

{\it \noindent Proof of the asymptotic representation \eqref{F2} of the minimal
polynomial \newline $M_n(x;W)$ on $E$.} Let $\rho_\nu$ be such a sequence of polynomials
whose existence is guaranteed by Lemma \ref{lemma-I}. Then in conjunction with
Lemma \ref{lemma-II} and Theorem \ref{thm1.6} it follows that for $\nu \geq \nu_0,$
uniformly on $E$
\begin{equation}\label{II-3-t0}
   2 \frac{M_n(x;W)}{W(x)} = {\mathcal R}_n^{+} (x;\rho_\nu) + {\mathcal R}_n^{-}
   (x;\rho_\nu) + o(1)
\end{equation}
where $o(1)$ is uniformly bounded with respect to $n$ and $\nu$ and
${\mathcal R}_n(z;\rho_\nu) := {\mathcal R}_n(z)$ is given by \eqref{t4}. First let
us demonstrate that for $\nu \geq \nu_0$ uniformly on $E$
\begin{equation}\label{II-3-t1}
   {\mathcal R}_n^{\pm} (x; {\rho}_{\nu}) =
   \frac{\phi^{\pm}(x;\infty)^n} {\prod\limits_{j=1}^{l-1}\phi^{\pm}(x,c_{j,n})}
   \sqrt{\frac{{\mathcal W}^{\pm}(x)}{{{\mathcal W}^{\mp}(x)} }}(1 + o(1)).
\end{equation}

Recall, see \cite{Wid}, that there exists a function $\Omega_\nu(z)$ which has
neither zeros nor poles, is analytic outside $E,$ and is such that $\Omega_\nu^{+}$
and $\Omega_\nu^{-}$ extend continuously to $E$ and satisfies for $\xi \in E$
\begin{equation}\label{II-3-t2}
   |\Omega_\nu(x)| = \sqrt{\Omega_\nu^{+}(x) \Omega_\nu^{-}(x)} = \rho_\nu(x) > 0
\end{equation}
Since
$$ f(z) = \frac{\rho_\nu(z)}{\omega_\nu(z)} \prod\limits_{j=1}^{\nu}
          \frac{\phi(z; w_{j,\nu})} {\phi(z; \infty)} $$
has neither zeros nor poles on $\bar{\mathbb C} \backslash E$ and
satisfies $|f^{\pm}| = 1$ on $E$ it follows as above that $\log |f|$ is a harmonic bounded
function on $\bar{\mathbb C} \backslash E$ which has a continuous extension to $E$ and
therefore $f \equiv 1,$ that is,
\begin{equation}\label{II-3-t3}
   \frac{\Omega_\nu(z)}{\rho_\nu(z)} = \prod\limits_{j=1}^{\nu} \frac{\phi(z;w_{j,\nu})}{\phi(z;\infty)}
\end{equation}

Note that
\begin{equation}\label{label-x}
 \sqrt{\frac{\Omega_{\nu}^{\pm}(x)}{\Omega_{\nu}^{\mp}(x)}} = \frac{\Omega_{\nu}^{\pm}(x)}{\rho_{\nu}(x)}
   = \prod\limits_{j=1}^{\nu} \frac{\phi^{\pm}(x;w_{j,\nu})}{\phi^{\pm}(x;\infty)}
\end{equation}

Next let us consider the function
\begin{equation}\label{II-4-4}
   I_\nu(z) = \exp \{ \frac{\sqrt{H(z)}}{2 \pi} \int_E \frac{1}{z-x} \log \frac{W(x)}{\rho_\nu(x)} \frac{dx}{h(x)}\}
\end{equation}
Because of \eqref{II-2-t3} $I_\nu(z)$ is analytic and nonzero on
$\bar{\mathbb C} \backslash E_l.$ Since $W/\rho_\nu \in C^{1 + \alpha}$
we may apply the Sochozki-Plemelj formula which yields that for $x \in E$
\begin{equation}\label{tt0}
   I_\nu^{\pm}(x) = \exp \{ \frac{1}{2} \log |\frac{W(x)}{\rho_\nu(x)}| \pm i \Psi(x) \},
\end{equation}
where
\begin{equation}\label{tt1}
   \Psi(x) = \frac{\sqrt{|H(x)|}}{2 \pi} \fpint{E}{} \frac{\log |W(t)/\rho_\nu(t)|}{x-t} \frac{dt}{\sqrt{|H(t)|}}.
\end{equation}
Moreover
\begin{equation}\label{II-4-5}
   |I_\nu(x)| = \sqrt{I_\nu^+(x) I_\nu^-(x)} = \frac{W(x)}{\rho_\nu(x)}
\end{equation}
Hence, recalling \eqref{II-3-t2}, the unique function ${\mathcal W}(z)$
with property \eqref{eq-x4}, introduced in the introduction, is given by
\begin{equation}\label{II-4-6}
   {\mathcal W}(z) = I_\nu(z) \Omega_\nu(z).
\end{equation}
Since by \eqref{II-1-t1}, \eqref{tt0} and \eqref{tt1}
\begin{equation}\label{II-4-7}
\begin{aligned}
   & I_\nu(z) \underset {\nu \to \infty} \longrightarrow 1
   {\rm \ uniformly \ on \ compact \ subsets \ of \ } \mathbb C \backslash E
   {\rm \ as \ well \ as \ } \\
   & I^{\pm}_\nu(x) \underset {\nu \to \infty} \longrightarrow 1 {\rm \ uniformly \ on \ } E\\
\end{aligned}
\end{equation}
it follows that
$$ {\mathcal W}^{\pm}(x) = I_\nu^{\pm}(x) \Omega_\nu^{\pm}(x) = \Omega_\nu^{\pm}(x) (1 + o(1))$$
hence
\begin{equation}\label{II-4-8}
  \sqrt{\frac{{\mathcal W}^{\pm}(x)}{{\mathcal W}^{\mp}(x)}} =
  \sqrt{\frac{\Omega^{\pm}_\nu(x)}{\Omega^{\mp}_\nu(x)}}(1+o(1))
\end{equation}
Thus by \eqref{t4} and \eqref{label-x} relation \eqref{II-3-t1} is proved.

Finally let us recall that because of \eqref{II-1-t1}
\begin{equation}\label{C1}
    c_{j,n}(\rho_\nu) = c_{j,n}(W)
\end{equation}
where the $c_{j,n}(\rho_\nu)$'s are the points \eqref{eq-xx3} and the $c_{j,n}(W)$'s
the points satisfying \eqref{F1}, using the fact that the associated Jacobi-inversion
problem is uniquely solvable. Since
\begin{equation}\label{702}
    \prod\limits_{j=1}^{l-1} \phi^{\pm}(x;c_{j,n}) =
    \prod\limits_{j=1}^{l-1} \phi^{\pm}(x;c_{j}) (1+o(1))
\end{equation}
we obtain by \eqref{II-3-t1}, \eqref{II-3-t0} and \eqref{eq-x4} the asymptotic
representation \eqref{F2} on $E.$
\qed
\newline

The asymptotic representation \eqref{F4} on $\bar{\mathbb C} \backslash E$
and the asymptotic value of the minimum deviation will be derived after the Lemma.

\begin{proof}[Proof of the asymptotic representation \eqref{F4} outside
of $E$:] Put $M_n(\cdot; W) = M_n,$ let $\tilde{M}_n$ be the polynomial
from \eqref{21-2} and set
\begin{equation}\label{II-6-x2}
    d_n = \frac{lc(M_n)}{lc(\tilde{M}_n)} = \frac{1}{lc(\tilde{M}_n)
    E_n(x^n;W)} = 1 + o(1)
\end{equation}
where the last equality follows by the fact, see \eqref{E11x2}, that
$$ \frac{1}{lc(\tilde{M}_n)} = E_{n-1}(x^n; \rho_\nu) + O(q^n) $$
and \eqref{LO} and \eqref{L1}.

Now let us consider
$$ \frac{M_n(z) - c_n \tilde{M}_n(z)}{\phi^n(z)} $$
Since $(M_n - d_n \tilde{M}_n)/\phi^n$ is a single valued function which
vanishes at $z=\infty$ we may apply Plemelj-Sochozki's formula and obtain,
$$
\begin{aligned}
     |\frac{(M_n - d_n \tilde{M}_n)(z)}{\phi^n(z)}|
      & = \left|\int_E \frac{(M_n - d_n \tilde{M}_n)(\xi)}{z-\xi}
        \left( \frac{1}{\phi^{+}(\xi)^n} - \frac{1}{\phi^{-}(\xi)^n}
        \right) d\xi \right| \\
      & = o(1) {\rm \ uniformly \ on \ compact \ subsets \ of \ }
        \bar{\mathbb C} \backslash E
\end{aligned}
$$
using the fact that by \eqref{II-3-t0} and \eqref{17.2}, \eqref{ten} and \eqref{M1}
$$ M_n(\xi) - d_n \tilde{M}_n(\xi) = o(1)  {\rm \ uniformly \ on \ } E.$$
Since $||\tilde{M}_n ||_E$ is bounded we know by the so-called Bernstein-Walsh Lemma
that $|\tilde{M}_n(z)| \leq const |\phi^n(z)|$ for
$z \in \bar{\mathbb C} \backslash E,$ hence
$$  \frac{M_n(z)}{\phi^n(z)} = \frac{\tilde{M}_n(z)}{\phi^n(z)} + o(1) =
    \frac{\Omega_\nu(z)}{\prod\limits_{j=1}^{l-1} \phi(c_{j,n};\infty)} + o(1)$$
which, by \eqref{II-4-6} and \eqref{II-4-7}, proves \eqref{F4}.
\end{proof}

\begin{proof}[Proof of Corollary \ref{cor1.3}:]
a) Since $1/ \prod\limits_{j=1}^{l-1} \phi(z;c_j)$ has a simple zero at
$c_j, j = 1,...,l-1,$ and is bounded from below on compact subsets of
$(a_{2j},a_{2j+1}) \backslash [c_j - \varepsilon, c_j + \varepsilon]$
it follows by \eqref{F4} that $M_n(x;1/W)$ has different sign on
$(a_{2j}+\varepsilon, c_j - \varepsilon)$ and
$(c_{j}+\varepsilon, a_{2j} - \varepsilon),$ $j=1,...,l-1,$ and thus at
least one zero in each $(c_{j} - \varepsilon, c_j + \varepsilon),$
and therefore exactly one zero. \\
b) The assertion follows by \eqref{t10}, \eqref{II-1-t2}, \eqref{C1},
\eqref{M1} and the relations $$ \sharp Z(P_{n+l-1},E_k) = \sharp
Z(\tilde{M}_n,E_k) = \sharp Z(M_n,E_k) $$ recalling the fact that at
the boundary points $|\tilde{M}_n/\rho_\nu|$ tends to one.
\end{proof}

\begin{proof}[Proof of Theorem \ref{thm2}]
The equivalence of statement a) and b) follows by Lemma \ref{lemma2.3} b)
and the equivalence of a) and c) by Corollary \ref{cor1.3}.
\end{proof}

\section{Proof of Theorem \ref{insertion}}

\begin{lemma}
Let $P_n$ be orthonormal on $E$ with respect to the weight function
$R/h \rho_{\nu} = R/ \sqrt{- H} \rho_\nu,$ $R / h \rho_\nu > 0$ on ${\rm int}(E)$
and $\rho_\nu > 0$ on $E$ and assume for simplicity of writing that the zeros
$w_j$ of $\rho_\nu$ are simple. Then the following relation holds
\begin{equation}\label{eq-e1t0}
   R P^2_n - S Q^2_m = 2 \rho_\nu g_{(n)}
\end{equation}
with
\begin{equation*}
   (R P_n)(w_j) = (\sqrt{H} Q_m)(w_j),
\end{equation*}
\begin{equation*}
   g_{(n)}(x) = \prod_{j = 1}^{l - 1} (x - x_{j,n}), {\rm \ where \ }
   x_{j,n} \in [a_{2 j},a_{2 j + 1}] {\rm \ for \ } j = 1,..., l - 1
\end{equation*}
and
\begin{equation*}
   R P_n (x_{j,n}) = \delta_{j,n} \sqrt{H} Q_m (x_{j,n})
\end{equation*}
where $\delta_{j,n} \in \{ \pm 1 \}.$ Furthermore putting
\begin{equation}
   {\mathcal R}_1 = \frac{R P^2_n}{\rho_\nu g_{(n)}} - 1
\end{equation}
and
\begin{equation}
    \sqrt{H} {\mathcal R}_2 = \frac{\sqrt{H} Q_m P_n}{ \rho_\nu g_{(n)} }
\end{equation}
we obtain
\begin{equation*}
   {\mathcal R}_1^2 - H {\mathcal R}_2^2 = 1
\end{equation*}
\begin{equation*}
   {\mathcal R}_1(x_{j,n}) = \delta_{j,n} \sqrt{H} {\mathcal R}_2(x_{j,n})
\end{equation*}
and
\begin{equation*}
   {\mathcal R}_1(w_j) = \sqrt{H} {\mathcal R}_2(w_j)
\end{equation*}
and there holds for $n \geq n_0,$ $k = 1,...,l- 1$
\begin{equation}\label{eq-inst1}
\begin{aligned}
   \sum_{j = 1}^{l - 1} \int_{\kappa^{*}_{j,n}}^{\kappa_{j,n}} \varphi_k
      & = \sum\limits_{j = 1}^{l - 1} \delta_{j,n} \int_{x_{j,n}^{-}}^{x_{j,n}^{+}} \varphi_k
        = - (2 n + 1 + \partial R - (\nu + l)) \int_{{\infty}^{-}}^{{\infty}^{+}} \varphi_k \\
      & - \sum\limits_{j = 1}^{\nu} \int_{w_j^{-}}^{w_j^{+}} \varphi_k
        + \sum_{j = 1}^{l - 1} ( 2 \sharp Z(P_n, E_j) + \sharp Z(R, E_j)) B_{k j}
\end{aligned}
\end{equation}
where ${\rm pr}(\kappa_{j,n}) = x_{j,n}$ and $\kappa_{j,n} \in {\mathcal R}^{\delta_{j,n}}.$

The $L_2$-minimum deviation is given by
\begin{equation}\label{eq-inst2}
\begin{aligned}
   \left( \int p^2_n \right) lc(\rho_\nu) = 2 (cap \ E)^{2 n + \partial R - (l - 1) - \nu}
   \frac{\prod_{j = 1}^{l - 1} \phi(\infty; x_{j,n})^{- \delta_{j,n}} }
        {\prod_{j = 1}^{\nu} \phi(\infty; w_j)}
   + O(q^n),
\end{aligned}
\end{equation}
where $0 < q < 1.$
\end{lemma}

\begin{proof}
The first statements follow by \cite{PehSIAM} and relation \eqref{eq-inst1}
follows by \cite[Lemma 3.1]{PehIMNR}. It has been shown in \cite[Lemma 2.3]{PehIMNR}
that
\begin{equation*}
    {\mathcal R}_1 = \frac{1}{2} \left( \psi_n + \frac{1}{\psi_n} \right) {\rm \ and \ }
    \sqrt{H} {\mathcal R}_2 = \frac{1}{2} \left( \psi_n - \frac{1}{\psi_n} \right)
\end{equation*}
where
\begin{equation*}
    \psi_n(z) = \phi(z,\infty)^{2 n + 1 + \partial R - (\nu + l)}
    \prod_{j = 1}^{\nu} \phi(z;w_j)
    \prod_{j = 1}^{l - 1} \phi(z,x_{j,n})^{\delta_{j,n}}
\end{equation*}
Thus we obtain by \eqref{eq-inst1} that
\begin{equation*}
\begin{aligned}
    & 1/ ( \int p^2_n ) lc(\rho_\nu) =
      \lim_{z \to \infty} \frac{1}{z^{2 n + 1 + \partial R - (\nu + l)}}
      \left( \frac{R P^2_n}{\rho_\nu g_{(n)}} - 1 \right) \\
  = & \lim_{z \to \infty} \frac{1}{z^{2 n + 1 + \partial R - (\nu + l)}}
      \frac{1}{2} \left( \psi_n + \frac{1}{\psi_n} \right) \\
  = & \frac{1}{2} (cap \ E)^{-(2 n + 1 + \partial R - (\nu + l))}
      \prod_{j = 1}^\nu \phi(z,w_j)
      \prod \phi(z,x_{j,n})^{\delta_{j,n}} + o(q^n),
\end{aligned}
\end{equation*}
$0 < q < 1,$ where we used the fact that
$\lim_{z \to \infty} 1/(\psi_n z^{2 n + 1 + \partial R - (\nu + l)}) = o(q^n).$
\end{proof}

\begin{lemma}\label{lem-new}
For every $\tilde{R}, R \in {\mathcal E},$ respectively, $\tilde{R}, R \in {\mathcal E}^{(1-x)}$
the solutions of \eqref{eq-inst1} satisfy
\begin{equation}\label{eq-zero}
    x_{j,n}(R) = pr(\kappa_{j,n}(R)) = pr (\kappa_{j,n}(\tilde{R})) = x_{j,n}(\tilde{R}) \ j = 1,...,l-1
\end{equation}
if ${\boldsymbol \kappa}_n(R),$ $ {\boldsymbol \kappa}_n(\tilde{R})$ $ \in $
${\sf{X}}_{j = 1}^{l - 1} (a_{2 j},a_{2 j + 1})^{-} \cup (a_{2 j},a_{2 j + 1})^{+}.$
In particular
\begin{equation*}
     g_{(n)}(x; R) = g_{(n)}(x; \tilde{R})
\end{equation*}
\end{lemma}

\begin{proof}
Since $\sum_{j = 1}^{l - 1} \sharp Z(R,E_j) B_{k j}/2$ and
$\sum_{j = 1}^{l - 1} \sharp Z(\tilde{R},E_j) B_{k j}/2$
differ modulo $2$ by half-periods only statement \eqref{eq-zero}
follows, see the proof of Lemma \ref{lemma2.2}, since the solutions of the
Jacobi inversion problem \eqref{eq-inst1} with respect to
$R$ and $\tilde{R}$ differ with respect to the sheet only.
\end{proof}

\begin{proof}[Proof of Theorem \ref{insertion}]
a) For ${\boldsymbol \sigma}_n := {\boldsymbol \sigma}_n(1/\rho_\nu)$
given by \eqref{insert_eq} there exists a $R({\boldsymbol \sigma}_n)$ such that
\begin{equation*}
   {\boldsymbol \sigma}_n = \sharp Z(R ({\boldsymbol \sigma}_n)) {\rm \ mod \ } 2
\end{equation*}
Now by the uniqueness of the real Jacobi inversion problem it follows that
\begin{equation*}
   {\boldsymbol c}_n = {\boldsymbol \kappa}_n(\sigma_n)
\end{equation*}
that is,
\begin{equation}\label{eq-t}
   c_{j,n} = x_{j,n} {\rm \ and \ } -1 = \delta_{j,n} \ \ j = 1,...,l-1
\end{equation}
where ${\boldsymbol \kappa}_n (\sigma_n)$ are the solutions from \eqref{eq-inst1}
and ${\boldsymbol c}_n$ the solution from \eqref{F1}. By Lemma \ref{lem-new} and
$\Phi(\infty,x) > 1$ for $x \notin E$ it follows that the RHS
takes its maximum for $\delta_{j,n} = -1,$ $j = 1,..., l-1,$ hence by \eqref{eq-t}
for $R(\sigma_n).$ Lemma \ref{lem-new} and Corollary \ref{} yield
\begin{equation*}
      \frac{1}{2} \left( \int p_n^2 \frac{R(\sigma_n)}{\hat{\rho}_\nu h} \right) =
      \frac{1}{2} E_{n - 1,2} (x^n; R(\sigma_n)/\hat{\rho}_\nu h)
  =   E_{2 n - 1, \infty}(x^{2 n}; 1/\hat{\rho}_\nu)
\end{equation*}
where $\hat{}$ means monic.

Now by the assumptions on the weight function $W,$ see e.g. \cite{ref6}, there
is a sequence of $\rho_\nu$'s positive on $E$ such that on $E$
\begin{equation*}
   \left| \frac{\rho_\nu(x)}{W(x)} - 1 \right| < \frac{const}{\nu} \omega_2\left( \frac{1}{\nu} \right)
\end{equation*}
which implies, as we have demonstrated above,
\begin{equation*}
   E_{n - 1, \infty}(x^n,1/\rho_\nu) = E_{n - 1, \infty}(x^n, W)(1 + o(1))
\end{equation*}
and, see \cite{Wid} or \cite{ref6}, that
\begin{equation}
   E_{n - 1,2}(x^n; W R(\sigma_n)/h) = E_{n - 1,2}(x^n; R(\sigma_n)/ h \rho_\nu)
\end{equation}
which gives the assertion.

b) Replacing $R(\sigma_n)$ from a) by $((1-x)R)(\sigma_n)$
the assertion follows analogously.
\end{proof}

It can be shown quite similarly as in the second part of the proof of Theorem \ref{thm1.6}
that the normalized minimal polynomial $M_{2 n}(\cdot; 1/\rho_\nu)$ with
$|| M_{2 n}(\cdot; 1/\rho_\nu) || = 1$ is given asymptotically by the orthogonal
polynomial as follows
\begin{equation*}
\begin{aligned}
     & \frac{R(x; \sigma_n) P^2_n(x; R(\sigma_n)/\rho_\nu h) - \rho_\nu(x) g_{(n)}(x;\sigma_n)}
            {\rho_\nu g_{(n)}(x;\sigma_n)} \\
   = & \frac{M_{2 n}(x;1/\rho_\nu) g_{(n)}(x,\sigma_n) + q(x)}{\rho_\nu(x) g_{(n)}(x;\sigma_n)}  \\
   = & \frac{M_{2 n}(x;1/\rho_\nu)}{\rho_\nu(x)} + O(r^n),
\end{aligned}
\end{equation*}
where $0 < r < 1,$ uniformly on compact subsets of ${\mathbb C} \backslash \{ c_1,...,c_{l-1} \}$
where $c_1,...,c_{l-1}$ are the accumulation points of zeros of $g_{(n)}(x,\sigma_n).$ For odd
$n$'s the assertion holds analogously.

\begin{proof}[Proof of Corollary \ref{cor1.7} ]
By \eqref{eq-sec2-x1} and \eqref{eq-sec2-x2}
$$ d \ln \phi(z;c) = - \sum\limits_{k=1}^{l-1} \omega_k(c) \varphi_k + \eta(z; c^{+}, c^{-}),$$
hence by \eqref{eq-xx5}
$$
\begin{aligned}
  \frac{\psi_n(z)}{{\mathcal W}(z)}
  & = \exp \{ \sum\limits_{k=1}^{l-1} ( - n \omega_k(\infty) + \sum\limits_{j=1}^{l-1} \omega_k(c_j) ) \int_a^z \varphi_k \}\\
  & \quad . \frac{ ( \exp  \int \eta(z; \infty^{+}, \infty^{-}) )^n }{ \prod\limits_{j=1}^{l-1} e^{\int \eta(z;c_j^{+},c_j^{-})} }
\end{aligned}
$$
Note that the first factor can be written by \eqref{F1} in terms
of integrals of $\log W.$ Now
$$ \prod\limits_{j=1}^{l-1} e^{\int \eta(z;c_j^{+},c_j^{-})} =
   \frac{ \vartheta (z; \sum_j \int_{a_1}^{c_{j,n}^{+}} \varphi_k) }
        { \vartheta (z; \sum_j \int_{a_1}^{c_{j,n}^{-}} \varphi_k) }$$
since the functions at the LHS and RHS have the same zeros and poles
and the same $\beta-$periods and coincide at the point $a_1.$ Analogously
the representation for $e^{\int \eta (z,\infty^{+},\infty^{-})}$ follows,
using \eqref{F1} with $W \equiv 1,$ or by taking a look at \cite[Proposition 2.1]{FalSeb}.
The assertion follows by \eqref{F1}.
\end{proof}

\end{document}